\documentclass[reqno,11pt]{amsart}
\usepackage[english]{babel}

\usepackage[utf8]{inputenc}
\usepackage[T1]{fontenc}

\usepackage[left=1.5in, right=1.5in]{geometry}

\usepackage{amsmath,amsthm,amssymb}

\usepackage{dsfont}   

\usepackage{xcolor} 
\usepackage{marginnote} 

\usepackage[hidelinks]{hyperref}

\usepackage[nameinlink]{cleveref}

\usepackage{stmaryrd} 

\usepackage{aliascnt} 

\theoremstyle{plain} 

\newtheorem{theorem}{Theorem}[section]

\newaliascnt{proposition}{theorem}
\newtheorem{proposition}[proposition]{Proposition}
\aliascntresetthe{proposition}

\newaliascnt{lemma}{theorem}
\newtheorem{lemma}[lemma]{Lemma}
\aliascntresetthe{lemma}

\newaliascnt{corollary}{theorem}
\newtheorem{corollary}[corollary]{Corollary}
\aliascntresetthe{corollary}

\newaliascnt{conjecture}{mainconj}
\newtheorem{conjecture}[conjecture]{Conjecture}
\aliascntresetthe{conjecture}

\newaliascnt{question}{mainconj}
\newtheorem{question}[question]{Question}
\aliascntresetthe{question}

\theoremstyle{definition} 
\newaliascnt{definition}{theorem}
\newtheorem{definition}[definition]{Definition}
\aliascntresetthe{definition}

\newtheoremstyle{uprightnote}
  {3pt}{3pt}{\normalfont}{}{\bfseries}{.}{ }{}
\theoremstyle{uprightnote}

\newaliascnt{remark}{theorem}
\newtheorem{remark}[remark]{Remark}
\aliascntresetthe{remark}

\newaliascnt{example}{theorem}

\aliascntresetthe{example}

\theoremstyle{uprightnote}
\newtheorem*{remarkunnumbered}{Remark}
\newenvironment{remark*}{\begin{remarkunnumbered}}{\end{remarkunnumbered}}

\newtheorem*{exunnumbered}{Example}
\newenvironment{example*}{\begin{exunnumbered}}{\end{exunnumbered}}

\crefname{theorem}{Theorem}{Theorems}
\crefname{proposition}{Proposition}{Propositions}
\crefname{lemma}{Lemma}{Lemmas}
\crefname{corollary}{Corollary}{Corollaries}
\crefname{conjecture}{Conjecture}{Conjectures}
\crefname{definition}{Definition}{Definitions}
\crefname{remark}{Remark}{Remarks}
\crefname{example}{Example}{Examples}
\crefname{question}{Question}{Questions}

\newcommand{\supp}{\operatorname{supp}}
\newcommand{\gen}{\operatorname{gen}}

\def\R{{\mathbb R}}
\def\Z{{\mathbb Z}}
\def\C{{\mathbb C}}
\def\D{{\mathbb D}}
\def\N{{\mathbb N}}

\def\S{{\mathbb S}}

\def \Q {{\mathbb Q}}

\def\T{{\mathbb T}}

\def\cH{{\mathcal H}}
\def\cP{{\mathcal P}}

\def\cA{{\mathcal A}}
\def\cL{{\mathcal L}}
\def\cI{{\mathcal I}}

\def\cV{{\mathcal V}}
\def\cK{{\mathcal K}}

\def\Xmt{(X,\mu,T)}

\def\cM{{\mathcal M}}

\def\cK{{\mathcal K}}

\def\cX{{\mathcal X}}

\newcommand{\norm}[1]{\left\lVert#1\right\rVert}

\newcommand{\RP}{\operatorname{RP}}
\newcommand{\nilbohrO}[1]{\operatorname{Nil_{#1}-Bohr_0}}

\newcommand{\bohr}{\operatorname{Bohr_0}}

\newcommand{\diff}{\mathop{}\!\mathrm{d}}

\newcommand{\hkbraket}[1]{\llbracket #1 \rrbracket}

\newcommand{\hknorm}[1]{{\left\vert\kern-0.25ex\left\vert\kern-0.25ex\left\vert #1 
    \right\vert\kern-0.25ex\right\vert\kern-0.25ex\right\vert}}

\newcommand{\1}{\ensuremath{\mathds{1}}}

\newcommand{\E}{\mathbb{E}}

\begin{document}

\title{Infinite sumsets in $U^k(\Phi)$-uniform sets}
\author{Trist\'an Radi\'c}
\address{Department of mathematics, Northwestern University, 2033 Sheridan Rd, Evanston, IL, United States of America}
\email{tristan.radic@u.northwestern.edu}
\thanks{ The author was partially supported by the National Science Foundation grant DMS-2348315.}
\subjclass[2020]{}

\begin{abstract}
    Extending recent developments of Kra, Moreira, Richter and Roberson, we study infinite sumset patterns in $U^k(\Phi)$-uniform subsets of the integers, defined via the local uniformity seminorms introduced by Host and Kra. We relate the degree $k$ of a $U^k(\Phi)$-uniform set to the existence of a rich variety of sumset patterns.  As a counterpart, we stablish higher order parity obstruction to sumsets arising from nilsystems.  We also provide examples of $U^k(\Phi)$-uniform sets for applications, including sets arising from the Thue-Morse and Rudin-Shapiro sequences. 
\end{abstract}

\maketitle

\tableofcontents

\section{Introduction}

Szemerédi's Theorem asserts that any subset of the natural numbers with positive density contains arbitrarily long arithmetic progressions. Different proofs of this theorem have played a central role in the development of additive combinatorics. In particular, since the introduction of Gowers norms in his proof of Szemerédi's Theorem \cite{gowers2001newszemeredi}, \emph{uniform sets} and related versions of \emph{pseudo-random sets} have been useful for finding finite patterns in subsets of the integers. For example, Gowers norms were used by Green and Tao to find arbitrarily long arithmetic progression in the set of prime numbers \cite{Green_Tao08} and Green, Tao and Ziegler established the connection to correlations with nilsequences by proving the related inverse theorem \cite{Green_Tao_Ziegler12}.

The ergodic counterpart, \emph{Host-Kra seminorms} defined for measure preserving systems and the structure theory of pronilfactors started in \cite{Host_Kra_nonconventional_averages_nilmanifolds:2005} have led to new recurrence properties (see \cite{Bergelson_Host_Kra05,donoso_koutsogiannis_Kuca_sun_tsinas2025resolvinghardy, frantzikinakis_host2018logarithmic_sarnak,frantzikinakis_kuca2025joint,Host_Kra05b} and for definition of pronilsystem see \Cref{sec preliminaries}). These recurrence results have combinatorial consequences via the Furstenberg correspondence principle that was originally introduced in \cite{Furstenberg77} to also provide an alternative proof of Szemerédi's Theorem.

More recently, the theory was used by Kra, Moreira, Richter and Roberson in \cite{kmrr_BB,kmrr1,kmrr25}, to established a series of infinite sumsets patterns in sets with positive density. This has led to numerous other related sumset results, see \cite{ackelsberg2024counterexamples, ackelsberg2025infinitepolynomial,ackelsberg_jamneshan2025equidistribution,DATI2025abelianBB,charamaras_mountakis2025amenable, hernandez2025infinite, host2019short, kousek2025asymmetric,  kousek_radic2025unrestrictedBB}. 
In \cite{kmrr25}, it is showed that for set of positive density $A \subset \N$ and $k\in \N$, there exist an infinite set $B \subset \N$ and an integer $t  \geq 0$ such that
\begin{equation} \label{eq KMRR BB}
    B \cup B^{\oplus 2} \cup \cdots \cup B^{\oplus k} \subset A - t, 
\end{equation}
where for $\ell \in \N$, $B^{\oplus \ell } = \{ b_1 + \cdots +  b_{\ell} \mid b_1,\ldots,  b_{\ell} \in B, b_1 < \cdots < b_{\ell}\}$.

Due to parity obstructions, the shift $t$ is needed in \eqref{eq KMRR BB}, take for instance the set of odd numbers. In this paper we stablish some higher-order parity obstructions arising from nilsystems and we show that $U^k(\Phi)$-uniform sets avoid such obstructions resulting in sumset patterns without shifts. 

To state the results precisely we introduce some notation and refer to \Cref{sec local unif semi} for full definitions.  Let $\norm{\cdot}_{U^k (\Phi)}$ denote the \emph{local uniformity seminorms} in $\ell^\infty(\N)$ introduced by Host and Kra in \cite{Host_Kra09}.
These seminorms are similar to Gowers norms but defined over $\Z$ rather that over a finite group and so can be handled more directly with dynamical systems tools using the original Host-Kra seminorms. For a set $A \subset \N$, we denote the \emph{density of $A$ along the F\o lner sequence} $\Phi$ as 
\begin{equation*}
    \diff_{\Phi}(A) =\lim_{N \to \infty} \frac{|A \cap \Phi_N |} { | \Phi_N|} 
\end{equation*}
whenever the limit exists.
For a given F\o lner sequence $\Phi$, we say that a set $A \subset \N$ is $U^k(\Phi)$\emph{-uniform} if $\diff_\Phi(A)$ is positive and $\norm{\1_A - \diff_{\Phi}(A) }_{U^k (\Phi)} =0$. Examples of $U^k(\Phi)$-uniform sets can be found in Sections \ref{sec constant length substitutions} and \ref{sec combinatorial app}.  
The degree $k$ should be understood as a \emph{level of randomness}, and we give precise consequences for how complex the sumset patterns can be contained in a $U^k(\Phi)$-uniform set.

 \begin{theorem} \label{theorem uniformity BB}
    Let $k \geq 2$ be an integer and $A_1, \ldots, A_k \subset \N$ be $U^k(\Phi)$-uniform sets with respect to the same F\o lner sequence $\Phi$. Then there exists an infinite set $B \subset \N$ such that
         \begin{equation*} \label{eq BB uniform}
              B^{\oplus i} \subset A_i  \ \text{ for all } \ i =1, \ldots, k.
         \end{equation*} 
\end{theorem}

  For a fixed $U^k(\Phi)$-uniform set $A \subset \N$, \cref{theorem uniformity BB} is a special case of Conjecture 3.26 stated in \cite{kra_Moreira_Richter_Roberson2025problems} and results from \Cref{sec nilsystem return times} are partial evidence for the full conjecture. Furthermore, our result lead to a more general version of the conjecture stated in \Cref{sec open q}.

The special case of \cref{theorem uniformity BB} for $k=2$ was proved in \cite{kra_Moreira_Richter_Roberson2025problems} using \emph{weak mixing sequences}. Also for $k=2$, a similar approach was used for infinite polynomial sumset configurations in \cite[Corollary 11.4]{ackelsberg2025infinitepolynomial}. 
Similarly, prior the results of \cite{kmrr1,Moreira_Richter_Robertson19},  Di Nasso, Goldbring, Jin, Leth, Lupini, and Mahlburg proved, among other things, that $U^2(\Phi)$-uniform sets contain the pattern $B_1 + B_2$ for infinite $B_1, B_2 \subset \N$. The techniques used in those proofs differ from the ones used here and, a priori, cannot be generalized to higher degrees because of the inductive process employed to construct the corresponding sumsets. 

  A larger variety of sumsets can be found in \Cref{sec uniform sets}. Among other result, we study sumsets along prescribed parallelepiped vertices (\Cref{sec leibman configuration and other linear}) and we also intersect $U^k(\Phi)$-uniform sets with $\nilbohrO{k-1}$ sets (\Cref{sec intersection nilbor}). For this last application we study measurable joinings of nilsystems. Furthermore, we strengthen \cref{theorem uniformity BB} for the case of $U^\infty(\Phi)$-uniform sets (that is $U^k(\Phi)$-uniform sets for all $k \in \N$), see \cref{U^infty very strong main theorem}.

 To capture new parity obstructions we introduce the notion of fully supported continuous disintegration (see \Cref{sec continuous disintegration}). This property is fulfilled by minimal pronilsystems and is preserved under natural operations between topological systems. It allows us to derive topological consequences from measure theory tools (see for example \cref{prop the support is contained}). 

Using the fully supported continuous disintegration, we derive sumset consequences for pronilsystem. In particular, \cref{theorem main dynamical very important} shows that under technical assumptions if two points $x,y \in X$ are in the same fiber over the $(k-1)$-step pronilfactor, then the set of return times of $x$ to any neighborhood of $y$ contain the pattern $B \cup B \oplus B \cup \cdots \cup B^{\oplus k}$ for some infinite $B \subset \N$. In the minimal case, we derive further implications using the regionally proximal relation of order $k$, $\RP^{[k]}(X)$, defined first by Host, Kra and Maass \cite{Host_Kra_Maass_nilstructure:2010}. The main result is summarized in \cref{theorem summary} and here we highlight two equivalences.

\begin{theorem} \label{theorem summary intro} 
    Let $k \geq 2$ be an integer, $(X, T)$  be a minimal pronilsystem. For $x,y \in X$, the following are equivalent
    \begin{itemize}
        \item  For every neighborhood $V$ of $y$ there exists $c_1, \ldots, c_k \in \N$ distinct natural numbers such that 
        $$\sum_{i \in I} c_i \in \{ n \in \N \mid T^n x \in V\} \quad \text{ for all } \quad I \subset \{1,\ldots, k\}, \ I \neq \emptyset; $$
        \item  For every neighborhood $V$ of $y$ there exists an infinite set $B \subset \N$ such that $$B \cup \cdots \cup B^{\oplus k} \subset \{ n \in \N \mid T^n x \in V\} $$ 
    
\end{itemize}

\end{theorem}

In \Cref{sec constant length substitutions}, we work with sequences arising from substitution systems (refer to that section for the definitions). We highlight two special cases: the Thue-Morse sequence  $(t(n))_{n \in \N_0} $ and the Rudin-Shapiro sequence $(r(n))_{n \in \N_0} $ given respectively by  $t(n) = ($number of digits $1$ in the binary expansion of $n) \mod 2$ and $r(n) = ($number of blocks $11$ in the binary expansion of $n) \mod 2$. The main sumset consequence in that context is the following.

\begin{theorem} \label{constant length subs sumset intro}
    Let $(X(\sigma),\mu,S)$ be a substitution subshift with $\sigma \colon \cA^* \to \cA^*$ primitive and constant length. Let $E \subset X(\sigma)$  be a clopen set, $x \in X(\sigma)$ and define  $\phi(n) = \1_E(S^n x)- \mu(E)$ for all $n \in \N$. If for every periodic sequence $(\psi(n))_{n \in \N}$
    \begin{equation} \label{eq orthogonal to periodic intro}
        \lim_{N \to \infty} \frac{1}{N} \sum_{n=1}^N \phi(n) \psi(n) =0,
    \end{equation}
    then $A = \{ n \in \N \mid S^n x \in E\}$ is $U^\infty(\Phi)$-uniform set. In particular, there exists a sequence of infinite sets $B_1 \supset B_2 \supset \ldots$ such that
    \begin{equation} \label{eq pattern un constant length subs}
        \bigcup_{k \in \N} B^{\oplus k}_k \subset A. 
    \end{equation}
\end{theorem}

As an application of the previous theorem, we recover sumset results proved by Bucci, Hindman, Puzynina and Zamboni in \cite{bucci_Hindman_Puzynina_Zamboni2013additive_thue_morse} for the Thue-Morse sequence: the set $\{ n \in \N \mid t(n) = a\}$ contains the pattern \eqref{eq pattern un constant length subs} for $a=0,1$. 
Likewise for the Rudin-Shapiro sequence $\{ n \in \N \mid r(n) = a\}$, $a=0,1$ (see \cref{prop sumset and uniformity Rudin-shapiro}). As an application, we provide a $U^\infty(\Phi)$-uniform set that is not an IP-set, in other words a set that does not contain the pattern $\bigcup_{k \in \N} B^{\oplus k}$ for an infinite set $B \subset \N$ (see \cref{prop infty uniform but not IP}). Also inspired by \cite{bucci_Hindman_Puzynina_Zamboni2013additive_thue_morse}, using a related substitution, we show that \cref{theorem summary intro} cannot be generalized to arbitrary minimal and uniquely ergodic systems (see also \cref{question RP}). 

\cref{constant length subs sumset intro} follows from the characterization of uniformity for sequences arising from primitive constant length substitution, see \cref{thrm uniformity constant length substitution}. Moreover, since Gowers norms are controlled by local uniformity seminorms \cite[Chapter 22, lemma 10]{Host_Kra_nilpotent_structures_ergodic_theory:2018}, \cref{constant length subs sumset intro} recovers the Gowers norm control for $k$-automatic sequences produced by a strongly connected, prolongable automaton proved in \cite[Theorem A]{byszewski_konieczny_mullner2020gowers_automatic}. Refer to  \cite{byszewski_konieczny_mullner2020gowers_automatic} for $k$-automatic sequence terminology and to Cobham's theorem \cite[Theorem 6.3.2]{cobham1972uniform,allouche_shallit2003automatic} for the relation between $k$-automatic sequences and constant-length substitutions. 
In particular, we extend results from \cite{konieczny2019gowers_Thue_Morse_Rudin_Shapiro} concerning Gowers uniformity of the Thue-Morse and Rudin-Shapiro sequences (see Propositions \ref{thrm infty uniformity thue-morse} and \ref{prop sumset and uniformity Rudin-shapiro}).

\subsection*{Notation}

By $\Z$, $\N_0$, $\N$, $\R$ and $\C$ we denote, respectively, the integers,
nonnegative integers, positive integers (natural numbers), the reals and the complex numbers. For a number $t \in \R$ we denote $e(t) = \exp(2 \pi i t)$ and $\S^1 = \{ z \in \C \mid |z| =1\}$.

\subsection*{Acknowledgment}
I am grateful to Bryna Kra for introducing me to this problem, for helpful guidance, and for later comments regarding previous drafts of this article. I would like to thank Xiandong Ye for his feedback that helped improve the manuscript. I benefited from the questions of the regular attendees of the informal seminar at Northwestern, and among them, I offer special thanks to Seljon Akhmedli, Redmond McNamara, and Andreas Mountakis for useful discussions. I thank Axel Álvarez and Felipe Hernández for suggestions on an early draft. Finally, concerning the second version of this paper, I express my gratitude once again to Bryna Kra and Axel Álvarez, as well as to Felipe Arbulú and Florian Richter.

\section{Preliminaries} \label{sec preliminaries}

In this section we fix notation and we gather classic definitions in the area. For more extended expositions refer to \cite{Furstenbergbook:1981,Glasner_ergodic_theory_joinings:2003,Host_Kra_nilpotent_structures_ergodic_theory:2018}.

\subsection{Dynamical systems}

A \emph{topological dynamical system} is a pair $(X, T)$ where $X$ is a compact metric space and $T \colon X \to X$ is a homeomorphism. A point $a \in X$ is \emph{transitive} if $\overline{\{ T^n a \mid n \in \N \}} =X$. The system $(X,T)$ is \emph{transitive} if it has a transitive point and $(X,T)$ is \emph{minimal} if every point is transitive. 

Let $\cM(X)$ be the set of Borel probability measures of $X$ and $\cM(X,T) \subset \cM(X)$ be the set of $T$-invariant Borel probability measures, that is if $\mu \in \cM(X,T)$ then $\mu(A) = \mu(T^{-1}A)$ for all Borel sets $A \subset X$. We say that $(X, \mu,T)$ is a \emph{measure preserving system} if $(X, T)$ is a topological dynamical system and $\mu \in \cM(X,T)$. With this definition all our measure preserving system have an underlying topological structure, and this is necessary for the sumset results.   
We say that $(X, \mu, T)$ is ergodic (or simply $\mu$ is ergodic) if whenever a measurable set $A \subset X$ satisfies $\mu(A \Delta T^{-1}A)=0$ then $\mu(A) = 0$ or $1$. 

We also consider other group actions, specifically $\Z^d$-actions that we denote $(X, \mu,$ $ T_1, \ldots, T_d)$ and cubic group actions (see \Cref{sec cubic space}). The previous definitions are extended naturally, for more detail refer to \cite{Glasner_ergodic_theory_joinings:2003}. 

 A \emph{F\o lner sequence} $\Phi =(\Phi_N)_{N \in \N}$ is a sequence of finite subsets of the natural numbers for which $ \lim_{N \to \infty} \frac{|\Phi_N \Delta (\Phi_N+t)|}{ | \Phi_N|} =0$ for every  $t \in \N$. If $(X, \mu,T)$ is a measure preserving system, we say that a point $a \in X$ is \emph{generic along a F\o lner sequence} $\Phi$, denoted $a \in \gen(\mu,\Phi)$ if for all continuous functions $f \in C(X)$
\begin{equation*}
    \lim_{N \to \infty} \frac{1}{|\Phi_N|} \sum_{n \in \Phi_N}f(T^na) = \int_X f \diff \mu.
\end{equation*}
By the mean ergodic theorem the set of generic points along a given F\o lner sequence has measure one whenever $\mu \in \cM(X,T)$ is ergodic. Also, if for some point $a \in X$, $\supp \mu \subset \overline{\{ T^n a \mid n \in \N\}}$ and $\mu $ is ergodic, then there exists a F\o lner sequence $\Phi$ such that $a \in \gen(\mu, \Phi)$.

 For a Borel map $\theta \colon X \to Y$ between topological spaces, if $\mu \in \cM(X)$ is a Borel probability measure, then \emph{the pushforward measure} $\nu \in \cM(Y)$ denoted $\nu = \theta \mu$ is given by $\nu(A) = \mu(\theta^{-1}(A))$ for all measurable sets $A \subset Y$. For two measure preserving systems $(X, \mu,T)$ and $(Y,\nu,S)$, we say that $(Y, \nu,S)$ is a \emph{factor} of $(X, \mu,T)$ if there exists a measurable map $\pi \colon X \to Y$, called a \emph{factor map}, such that $\pi\mu = \nu$ and $\pi \circ T = S \circ \pi$, up to null-sets. Abusing notation sometimes we just say that $Y$ is a factor of $X$ and that $X$ is an \emph{extension} of $Y$. 
 If the factor map $\pi \colon X \to Y$ is continuous, surjective and $(\pi \circ T)(x) = (S \circ \pi)(x)$ for all $x \in X$, we say that $\pi$ is a \emph{continuous factor map} and $Y$ is a \emph{topological factor} of $X$. We abuse notation by using the same letter $T$ for the transformation in $X$ and in its factor $Y$.
 
For a factor map $\pi: X\to Y$ and a function $f\in L^2(\mu)$, 
$\mathbb{E}(f | Y)$ denotes the conditional expectation 
$\mathbb{E}(f | \pi^{-1}\mathcal{B}(Y))$, where $\mathcal{B}(Y)$ is the Borel 
$\sigma$-algebra on $Y$. Notice that $\mu$-almost surely, $\mathbb{E}(f | Y)(x) = \mathbb{E}(f | Y)(x')$ whenever $\pi(x) = \pi(x') = y$, so abusing notation we also write $\mathbb{E}(f | Y)(y)$ and we think $\mathbb{E}(f | Y)$ as a function in $L^2(\nu)$. 

Consider a sequence of measure preserving systems $(X_j,\mu_j,T)_{j\in \N}$ for which there is a factor map $\theta_{j}: X_{j+1}\to X_j$ for each $j\in \N$. $(X,\mu,T)$ is the \emph{inverse limit} of this sequence if for each $j \in \N$ there exists a factor map $\pi_j \colon X \to X_j$ such that
\begin{itemize}
    \item  $\pi_{j+1}=\theta_{j} \circ \pi_j$, for every $j\in \N$ and 
    \item $\bigcup_{i\in \N} \{f \circ \pi_i : f\in L^2(\mu_i)\}$ is dense in $L^2(\mu)$.
\end{itemize}
A similar definition applies for topological dynamical systems (replacing factor maps by continuous factor maps and $f \in L^2(\mu_j)$ by $f \in C(X_j)$). 

\subsection{Pronilfactors and Host-Kra seminorms}  \label{sec pronilfactors and uniformity seminorms}

A $k$-\emph{step} \emph{nilmanifold} $X = G/ \Gamma$ is the quotient space of a $k$-step nilpotent Lie group $G$ with $\Gamma$ a co-compact discrete subgroup. We endow $ G/ \Gamma$ with its Haar measure $m_X$ and let $T \colon X \to X$ be the left translation by a fixed element $\tau \in G$. In that case we say that $(X, m_X, T)$ is a $k$-\emph{step nilsystem}. We can generalize it to $(X, m_X, T_1, \ldots, T_d)$ where $\tau_1, \ldots, \tau_d \in G$ commute and $T_i \colon X \to X$ are their respective left translation maps. Without loss of generality, we always suppose that $G$ is generated by the connected component of the identity $G^0$ and the points $\tau_1, \ldots, \tau_d$. 

An inverse limit of $k$-step nilsystems is called a $k$-step \emph{pronilsystem}. An $\infty$-step \emph{pronilsystem} is 
an inverse limit of a sequence $\left((Z_k, m_k,T)\right)_{k\in \N}$, where for each $k \in \N$, $(Z_k, m_k, T)$ is a $k$-step pronilsystem. We say that a system is a pronilsystem if it is a $k$-step pronilsystem for some $k \in \N \cup \{ \infty\}$. Every pronilsystem is \emph{distal}, that is, for any $x,y \in X$ such that $\inf_{n \in \N} d(T^n x, T^n y)=0$, then $x=y$.   

Pronilsystems $(X,T_1, \ldots, T_d)$ have the useful property that the topological notions of transitivity, minimality and unique ergodicity are equivalent. Moreover, if $(X, T_1, \ldots, T_d)$ is a pronilsystem (non necessarily minimal), then for any $x \in X$, $\overline{\{T_1^{n_1}\cdots T_d^{n_d}x \mid (n_1, \ldots, n_d) \in \Z^d\}} \subset X$ is minimal and uniquely ergodic, see \cite{leibman2005pointwise}.

Let $(X, \mu,T)$ be an ergodic measure preserving system. For every $k \in \N \cup \{\infty\}$, $(X, \mu,T)$ admits a maximal factor isomorphic to a $k$-step pronilsystem, called the $k$\emph{-step pronilfactor}. We denote this factor $Z_k(X)$, $Z_k(\mu)$ or just $Z_k$ depending on the context. In what follows, we always consider the topological model of 
 $Z_k$ which is a 
$k$-step pronilsystem. For $k=1$, this factor is also known as the \emph{Kronecker factor} and can be characterized as the factor spanned by the eigenvalues of $\Xmt$, that is the functions $f \in L^2(\mu)$ such that $Tf = \lambda f$ for some $\lambda \in \S^1$. An ergodic system $\Xmt$ is \emph{weak mixing} if $\lambda=1$ is the unique eigenvalue. If $\Xmt$ is weak mixing then $Z_k(\mu)$ is the trivial factor for all $k \in \N$. Similarly, if for some $k \in \N$, $Z_{k+1} = Z_k$, then $Z_s = Z_k$ for all $s \geq k$.

    For a measure preserving system $(X, \mu,T)$, $f \in L^{\infty}(\mu)$ and $k \in \N_0$, the \emph{$k$-th Host-Kra seminorm of} $f$, $\norm{f}_{U^{k}(X,\mu,T)}$, is defined inductively as follows: 

    \begin{align}
        \norm{f}_{U^0(X, \mu,T)}& =  \int_X f \diff \mu  \quad \text{ and} \nonumber \\
         \norm{ f}_{U^{k+1}(X, \mu,T)}^{2^{k+1}}&= \lim_{H \to \infty }\frac{1}{H} \sum_{h \leq H} \norm{f \cdot \overline{T^hf}}_{U^k(X, \mu,T)}^{2^k}. \label{eq def seminorm}
    \end{align}
    The function $\norm{ \cdot }_{U^{k}(X, \mu,T)}$ is a seminorm when $k \geq 1$. The fact that the limit in \eqref{eq def seminorm} exists is proved in \cite{Host_Kra_nonconventional_averages_nilmanifolds:2005} by iterative applications of the mean ergodic theorem. 
The main theorem in \cite{Host_Kra_nonconventional_averages_nilmanifolds:2005} shows that the $k$-step pronilfactor, $Z_{k}$, of an ergodic system $(X, \mu,T)$ is characterized by the property 
    \begin{equation} \label{eq def Z_k con seminorms}
        \norm{f}_{U^{k+1}(X, \mu,T)} = 0 \iff \E(f \mid Z_{k}) =0
    \end{equation}
    for any $f\in L^{\infty}(\mu)$. Similarly,
    \begin{equation*}
           \E(f \mid Z_\infty)=0 \iff \lVert f \rVert_{U^k(X,\mu,T)} = 0 \quad \text{ for all } k \in \N.
        \end{equation*}

    \begin{definition}
        We say that a measurable set $E \subset X$ is $U^k(X, \mu,T)$-uniform for some $k \in \N$, if $\mu(E) >0$ and $\norm{\1_E - \mu(E)}_{U^k(X, \mu,T)}=0$. Using \eqref{eq def Z_k con seminorms}, this is equivalent to $\E(\1_E \mid Z_{k-1}) (x)= \mu(E)$ for $\mu$-almost every $x \in X$. 
        Similarly, $E \subset X$ is $U^\infty(X, \mu,T)$-uniform if it is $U^k(X, \mu,T)$-uniform for all $k \in \N$. 
    \end{definition}




    
    \subsection{Local uniformity seminorms} \label{sec local unif semi}

  We say that a bounded function $\phi \colon \N \to \C$ \emph{admits correlations along the F\o lner sequence} $\Phi$ if the limit
\begin{equation*}
    \lim_{N \to \infty} \frac{1}{|\Phi_N|} \sum_{n \in \Phi_N} \prod^k_{i=1} C^{\epsilon_i} \phi(n + h_i)
\end{equation*}
exists for all $h_1, . . . , h_k \in \N$ and $\epsilon_1, \ldots, \epsilon_k \in \{0,1\}$, where for a complex number $z$, $C z = \overline{z}$ is the complex conjugation.  Assuming that $\phi \colon \N \to \C$ admits correlations along a given F\o lner sequence $\Phi$ we define, inductively in $k$, the \emph{local uniformity seminorms of $\phi$ along $\Phi$} by
\begin{align}
    &\norm{\phi}_{U^0(\Phi)} = \lim_{N \to \infty} \frac{1}{|\Phi_N|} \sum_{n \in \Phi_N} \phi(n) \nonumber \\
    &\norm{\phi}_{U^{k+1}(\Phi)}^{2^{k+1}} = \lim_{H \to \infty}\frac{1}{|H|} \sum_{h =1}^H \norm{\Delta_h \phi }^{2^k}_{U^k(\Phi)}  \label{eq local uniformity norms} 
\end{align}
where $(\Delta_h \phi)(n) = \phi(n) \overline{\phi(n+h)}$ for all $n,h \in \N$. The existence of the limits in \eqref{eq local uniformity norms} is proved by Host and Kra in \cite{Host_Kra09}, using the Host-Kra seminorms defined in \eqref{eq def seminorm}. 
   As it is mentioned in the introduction a set $A \subset \N$ is $U^k(\Phi)$\emph{-uniform} if the indicator function $\1_A \colon \N \to \C$ admits correlations along $\Phi$, $\diff_{\Phi}(A) > 0$ and $\norm{\1_A - \diff_{\Phi}(A) }_{U^k (\Phi)} =0$.

    Notice that for any F\o lner sequence $\Phi$ and any bounded sequence $\phi \in \ell^\infty (\N)$ there exists a subsequence $\Psi$ of $\Phi$ for which $\phi$ admits correlations.
  We also notice some similarities with the Host-Kra seminorms, for $k \in \N$, $f \in L^\infty(\mu)$ in a measure preserving system $\Xmt$ and $\phi \in \ell^\infty (\N)$ admitting correlations along $\Phi$
  \begin{equation*}
      \norm{f}_{U^k\Xmt} \leq \norm{f}_{U^{k+1}\Xmt} \quad \text{ and } \quad  \norm{\phi}_{U^k(\Phi)} \leq \norm{\phi}_{U^{k+1}(\Phi)}. 
  \end{equation*}
  In \Cref{sec furstenberg correspondence} we make the connection between this two seminorms explicit following \cite{Host_Kra09}. Similarly, we show the correspondence between $U^k(X, \mu,T)$-uniform sets $E \subset X$ of a measure preserving system and  $U^k(\Phi)$-uniform sets $A \subset \N$, see \cref{furstenberg correspondence for sets}.

\subsection{Cubic spaces} \label{sec cubic space}

For $k \in \N$, we write $\hkbraket{k} = \{0,1\}^k$ and $\hkbraket{k}^* = \hkbraket{k} \setminus \{\underline{0}\}$, where $\underline{0}$ is the vector $0$. Also, for a vertex $\epsilon = \epsilon_1 \epsilon_2 \cdots \epsilon_k \in \hkbraket{k} $, we write $|\epsilon|= \epsilon_1 + \cdots + \epsilon_k$.   For a set $X$ and functions $f,g \colon X \to \C$, we denote $f \otimes g \colon X \times X \to \C$ the function given by $(f \otimes g) (x,y) = f(x) g(y)$ for all $x,y \in X$. 

Following \cite[section 4]{kmrr1} for $\boldsymbol x \in X^{\hkbraket{k}} = X^{2^k}$ and $i =1,\ldots, k$ we write $F_i \boldsymbol x$ and $F^i \boldsymbol x$ its $i$th \emph{lower} and \emph{upper} face define by
\begin{align} \label{eq lower and upper faces}
    F_i\boldsymbol{x} =  (x_\epsilon \colon \epsilon_i =0) \quad \  F^i\boldsymbol{x} =  (x_\epsilon \colon \epsilon_i =1)
\end{align}
and then formally $(F_i\boldsymbol{x}, F^i\boldsymbol{x}) \in X^{\hkbraket{k}}$. In particular $X^{\hkbraket{k}}$ is identified with $X^{\hkbraket{k-1}} \times X^{\hkbraket{k-1}}$ with $\boldsymbol{x} = (F_1 \boldsymbol{x}, F^1 \boldsymbol{x})$.
 In \cite{Host_Kra_nonconventional_averages_nilmanifolds:2005}, Host and Kra introduced the following measure, defined inductively in $X^{\hkbraket{k}}$ as follows, $\mu^{\hkbraket{0}}= \mu$ and for $G_0, G_1 \in L^\infty(\mu^{\hkbraket{k}})$
\begin{align*} 
    \int_{X^{\hkbraket{k+1}}} G_0 \otimes G_1 \diff \mu^{\hkbraket{k+1}} = \int_{X^{\hkbraket{k+1}}} \E(G_0 \mid \cI(T^{\hkbraket{k}})) \otimes \E(G_1 \mid \cI(T^{\hkbraket{k}})) \diff \mu^{\hkbraket{k}} 
\end{align*}
where $\cI(T^{\hkbraket{k}})$ is the $\sigma$-algebra of $T^{\hkbraket{k}}$-invariant sets for $T^{\hkbraket{k}} = T \times \cdots \times T$. 

The measure $\mu^{\hkbraket{k}}$ is invariant under digit permutation that is $\phi\mu^{\hkbraket{k}}= \mu^{\hkbraket{k}}$ where $\phi$ is an arbitrary permutation of elements in $\{1, \ldots, k\}$ which induces a map $\phi \colon \hkbraket{k} \to \hkbraket{k}$ by $(\phi\epsilon)_i \mapsto \epsilon_{\phi(i)} $ for all $i =1, \ldots, k$ and thus induces a map, also denoted by $\phi \colon X^{\hkbraket{k}} \to X^{\hkbraket{k}}$ given by $(\phi x)_\epsilon = x_{\phi(\epsilon)}$ for all $\epsilon \in \hkbraket{k}$. 
From \cite{Host_Kra_nonconventional_averages_nilmanifolds:2005}, we have that for $f \in L^{\infty}(\mu)$, 
\begin{equation*}
    \norm{f}_{U^k(X, \mu, T)} = \bigg( \int_{X^{\hkbraket{k}}} \bigotimes_{\epsilon \in \hkbraket{k}} C^{|\epsilon|} f \diff \mu^{\hkbraket{k+1}}\bigg)^{1/2^k}.
\end{equation*}
Let $Q^{\hkbraket{k}}_0(T)$ be the group generated by the upper face transformations $S_1, \ldots, S_k$ where for $\boldsymbol{x} \in X^{\hkbraket{k}}$ and a given $i \in \{1, \ldots,k\}$,
\begin{equation*}
    (S_i \boldsymbol{x})_\epsilon = \begin{cases}
        \begin{array}{cc}
            Tx_\epsilon & \text{ if } \epsilon_i =1  \\
            x_\epsilon & \text{ if } \epsilon_i =0.
        \end{array}
    \end{cases}
\end{equation*}
In other words, $S_i(F_i\boldsymbol{x}, F^i\boldsymbol{x}) = (F_i\boldsymbol{x}, T^{\hkbraket{k-1}} F^i\boldsymbol{x})$. The measure $\mu^{\hkbraket{k}}$ is invariant under $Q^{\hkbraket{k}}_0(T)$ and therefore it is invariant under the group $Q^{\hkbraket{k}}(T) = \langle Q^{\hkbraket{k}}_0(T), T^{\hkbraket{k}}\rangle$. Moreover, whenever $\mu$ is ergodic, $\mu^{\hkbraket{k}}$ is ergodic under the action of $Q^{\hkbraket{k}}(T)$. 
 For a topological system $(X,T)$, we write 
$$Q^{\hkbraket{k}}(X) = \overline{\{ S (x, \ldots, x) \mid x \in X, S \in Q^{\hkbraket{k}}(T)\}} .$$

Similarly, for an $s$-step minimal pronilsystem $(Z,m ,T)$ and a point $z \in Z$, 
\begin{equation} \label{eq cubic with a fixed point}
    Q_{z}^{\hkbraket{k}}(Z) = \{ \boldsymbol z \in Q^{\hkbraket{k}}(Z) \mid z_{\underline{0}} = z \} = \overline{ \{ S (z, \ldots, z) \mid S \in Q_0^{\hkbraket{k}} (T) \}}
\end{equation}
has unique invariant measure, where the fact that the second equality holds true for any minimal distal system can be found in \cite[Proposition 4.2]{Host_Kra_Maass_nilstructure:2010}.

\subsection{Erd\H{o}s cubes and Erd\H{o}s progressions} Following \cite{kmrr1,kmrr25}, for a measure preserving system $(X,\mu,T)$, we say that 
    \begin{itemize}
        \item $(x_0, x_1, \ldots, x_k) \in X^{k+1}$ is an Erd\H{o}s progression if there is a sequence $(c(n))_{n \in \N}$ such that 
        \begin{equation*}
            \lim_{n \to \infty} (T \times \cdots \times T) ^{c(n)} (x_0, \ldots, x_{k-1}) = (x_1, \ldots, x_k)
        \end{equation*}
        \item $\boldsymbol x \in X^{\hkbraket{k}}$ is an Erd\H{o}s cube if there are sequences $(c_1(n))_{n \in \N}, \ldots, (c_k(n))_{n \in \N}$ such that 
        \begin{equation*}
            \lim_{n \to \infty} (T^{\hkbraket{k-1}})^{c_i(n)} (F_i \boldsymbol{x}) = F^i \boldsymbol x.
        \end{equation*}
        where $F_i \boldsymbol{x} $ and $F^i \boldsymbol{x}$ are defined in \eqref{eq lower and upper faces}.
    \end{itemize}

We denote $E^{(k)}_a(X) \subset X^{k+1}$, (resp. $E^{\hkbraket{k}}_a(X) \subset X^{\hkbraket{k}}$) the set of Erd\H{o}s progressions (resp. Erd\H{o}s cubes) starting at a point $a \in X$. By \cite[Lemma 5.16]{kmrr1}, $E^{\hkbraket{k}}(X) \subset Q^{\hkbraket{k}}(X)$. Notice that one can turn every Erd\H{o}s progression in an Erd\H{o}s cube, to see that we first define the $k$-vertex embedding

\begin{definition} \label{def ell-vertex}
    The $k$-vertex embedding is the function $\psi_k \colon X^{k+1} \to X^{\hkbraket{k}}$ given by
\begin{equation*}
    (\psi_k(x_0, \ldots,x_k))_{\epsilon} = x_{|\epsilon|} \quad \text{ for all } \epsilon \in \hkbraket{k}.
\end{equation*}
\end{definition}

\begin{lemma} \label{lemma erdos progression vs erdos cube}
    Let $x=(x_0, x_1 , \ldots, x_{k-1})$ be an Erd\H{o}s progression, then $\psi_{k}(x)$ is an Erd\H{o}s cube. Conversely if $\boldsymbol x$ is an Erd\H{o}s cube and $x_\epsilon = x_{\epsilon'}$ whenever $|\epsilon| = | \epsilon'|$, then $(x_{0^k}, x_{0^{k-1}1}, x_{0^{k-2}11}, \ldots, x_{1^k})$ is an Erd\H{o}s progression. 
\end{lemma}

In particular, $\psi_{k}(E^{(k)}_a(X)) = E^{\hkbraket{k}}_a(X) \subset Q^{\hkbraket{k}}(X)$. 

\begin{proof}
    The first inclusion is clear, indeed from the definition of Erd\H{o}s progression we take $(c_i(n))_{n \in \N} = (c(n))_{n \in \N}$ for all $i = 1, \ldots, k$. 
    
    For the converse let $ \boldsymbol x \in E^{\llbracket k \rrbracket}(X)$ for which $x_\epsilon = x_{\epsilon'}$ whenever $|\epsilon| = | \epsilon'|$ and denote $(c(n))_{n \in \N} = (c_k (n))_{n \in \N} \subset \N$ the sequences such that $\lim_{n \to \infty} (T^{\hkbraket{k-1}})^{c(n)} (F_k \boldsymbol{x}) = F^k \boldsymbol{x}$. In particular, 
    \begin{equation*}
       \lim_{n \to \infty} T^{c(n)} x_{0^k} = x_{0^{k-1}1}.
    \end{equation*}
    By definition of Erd\H{o}s cube we then get that for the same sequence $(c(n))_{n \in \N}$, $\displaystyle \lim_{n \to \infty} T^{c(n)} x_{0^{k-2}10} = x_{0^{k-2}11}$, 
    and by assumption $x_{0^{k-2}10} = x_{0^{k-1}1}$, so $ \displaystyle \lim_{n \to \infty} T^{c(n)} x_{0^{k-1}1} = x_{0^{k-2}11}$. We then proceed similarly using at each step that $x_{0^{k-1-j}1^{j}0} = x_{0^{k-j}1^{j}}$. 
\end{proof}

The connection between Erd\H{o}s cubes and progressions to sumsets was stablish by Kra, Moreira, Richter and Robertson, first in \cite{kmrr1} and later in \cite{kmrr_BB,kmrr25}.

\begin{lemma}[Lemma 2.2 \cite{kmrr25} and Corollary 4.6 \cite{kmrr25}] \label{Erdos cubes and progressions and sumsets} 

    Let $a \in X$ be a point in a measure preserving system $(X, \mu,T)$.
    \begin{enumerate}
        \item Suppose $E_1, \ldots, E_k$ are open sets, if there exists an Erd\H{o}s progression $(a,x_1, \ldots,x_k)$ with $x_i \in E_i$, for all $i= 1, \ldots, k$ then  there exists an infinite set $B \subset \N$ such that $B^{\oplus i} \subset \{ n \in \N \mid T^n a \in E_i\}$ for all $i = 1, \ldots, k $.
        \item Suppose $E_\epsilon$ are open sets for $\epsilon \in \hkbraket{k}^*$, if there exists an Erd\H{o}s cube $\boldsymbol{x}$ with $x_{0^k} =a$ and $x_\epsilon \in E_\epsilon$ for all $\epsilon \in \hkbraket{k}^*$, then there are infinite sets $B_1, \ldots, B_k \subset \N$ such that $\sum_{i=1}^k \epsilon_i B_i \subset \{ n \in \N \mid T^n a \in E_\epsilon\}$ for all $\epsilon \in \hkbraket{k}^*$.
    \end{enumerate}
\end{lemma}

In \cite{kmrr1}, the authors require to the system they are working with to have continuous map to there pronilfactors. This is not a restrictive assumption because any ergodic system has a topological extension that fulfills that property (see \cite[Lemma 5.8]{kmrr1}).

\begin{definition} \label{def cont pronilfactor}
    Let $(X, \mu,T)$ be an ergodic system and,  for every $k \in \N$, let $(Z_k, m_k, T)$ be its pronilfactor. We say that $(X, \mu,T)$ has \emph{topological pronilfactors} if for all $k \in \N$ the factor map $\pi_k \colon X \to Z_k$ is a topological factor map.  
\end{definition}

 In \cite[Proposition 4.3]{hernandez_kousek_radic2025density_Hindman} we remark that if $\pi_k \colon X \to Z_k$ is continuous for every $k \in \N$ then $\pi_\infty \colon X \to Z_\infty$ is also continuous. 
    We recall that in this paper, for $\Xmt$ and $k \in \N$, we denote by $Z_k$ the maximal $k$-step pronilsystem, that is $Z_k$ is a (topological) inverse limit of $k$-step nilsystems. That said, when we say that $\pi_k\colon X\to Z_k$ is a topological factor in \cref{def cont pronilfactor}, we are fixing $Z_k$ to be the pronilsystem, rather than an arbitrary measurably isomorphic model. This is crucial as it is illustrated \cite{Durand&Frank&Maass:2019} and \cite[Example 6.1]{glasner_Huang_Shao_Ye2020regionally} for $Z_1$, and more recently for general $k \in \N$ in  \cite{akhmedli2026top_vs_measurable_nil} in order to distinguish between $Z_k$ and $X/\RP^{[k]}(X)$ for minimal systems, see \Cref{sec RPk} for the definition of $\RP^{[k]}(X)$. 

To find Erd\H{o}s cubes and progressions we define measures constructed as lift of measures in pronilsystems. For the definition of the lift of a measure see \Cref{sec useful lemmas}. Following \cite{kmrr25} we define

\begin{definition} \label{definition of sigma}
    Let $(X, \mu,T)$ an ergodic system with topological pronilfactor. For $k \in \N$ fixed, $\pi_{k-1} \colon X \to Z_{k-1} $ the continuous factor map and $a \in \gen(\mu,\Phi)$ for some F\o lner sequence $\Phi$, we denote
    $\xi_a \in \cM(Z_{k-1}^k)$ the unique invariant measure of $\overline{\{(T \times T^2\times \cdots \times T^k)^n (\pi_{k-1}(a), \ldots, \pi_{k-1}(a)) \mid n \in \Z \} }$. The measure $\sigma_a \in \cM(X^{k+1})$ equals $\delta_a \times \tilde \sigma_a$ where $\delta_a$ is the Dirac measure of $a$ and $\tilde \sigma_a$ is the lift of $\xi_a$. 
\end{definition}

In \cite{kmrr25} they called that measure \emph{progressive}. 

\begin{definition} \label{definition of sigma cubic}
    Let $(X, \mu,T)$ an ergodic system with topological pronilfactor. For $k \in \N$ fixed, $\pi_{k-1} \colon X \to Z_{k-1} $ the continuous factor map and $a \in \gen(\mu,\Phi)$ for some F\o lner sequence $\Phi$, we denote
    $\xi_a^{\hkbraket{k}} \in \cM(Z_{k-1}^{\hkbraket{k}})$ the unique invariant measure of $Q_{\pi_{k-1}(a)}^{\hkbraket{k}}(Z_{k-1})$ (see \eqref{eq cubic with a fixed point}) and $\xi_a^{\hkbraket{k}^*}$ the projection $\xi_a^{\hkbraket{k}}$ to the coordinates $\hkbraket{k}^*$. The measure $\sigma_a^{\hkbraket{k}} \in \cM(X^{\hkbraket{k}})$  equals $\delta_a \times \tilde \sigma_a^{\hkbraket{k}}$ where $\delta_a$ is the Dirac measure of $a$ and $\tilde \sigma_a ^{\hkbraket{k}}$ is the lift of $\xi_a^{\hkbraket{k}^*}$.  
\end{definition}



\begin{lemma}[Theorem 4.9 \cite{kmrr25} and Theorem 7.5 \cite{kmrr1}] \label{lemma characterization progressive measures} \label{lemma sigma^[k] property} 
    Let $(X, \mu,T)$ be an ergodic system with topological pronilfactor, $a \in \gen(\mu,\Phi)$ for some F\o lner sequence $\Phi$, $k \in \N$. 
    \begin{enumerate}
        \item If $\sigma_a \in \cM(X^{k+1})$ as in \cref{definition of sigma}, then the set of Erd\H{o}s progressions $E^{(k)}_a(X)$ is dense in $\supp \sigma_a$.
        \item If $\sigma_a^{\hkbraket{k}} \in \cM(X^{\hkbraket{k}})$ as in \cref{definition of sigma cubic}, then the set of Erd\H{o}s cubes $E^{\hkbraket{k}}_a(X)$ is dense in $\supp \sigma_a^{\hkbraket{k}}$. 
    \end{enumerate}
  
\end{lemma}

Notice that by the previous lemma, if $U_1, \ldots, U_k \subset X$ are open sets such that $\sigma_a(X \times U_1 \times \cdots \times U_k)>0$, then there exists an Erd\H{o}s progression $(a, x_1, \ldots, x_k)$ such that $x_i \in U_i$ for all $i=1, \ldots, k$. Similarly for Erd\H{o}s cubes.

The consequence of \cite[Theorem 7.5]{kmrr25} is actually a $\sigma_a^{\hkbraket{k}}$-almost surely statement, but for our purposes it suffices with the density result. For the progressive measure case, the $\sigma_a$-almost surely result remains open except for the Kronecker case, that is for $k=2$, that was proved in \cite{kmrr_BB}.

We finish this section by noticing that in \cite[Section 7]{kmrr1}, the original definition of $\sigma_a^{\hkbraket{k}}$ differs from \cref{definition of sigma cubic} but it is equivalent. The equality is a direct consequence of the unique ergodicity of $Q_{\pi_{k-1}(a)}^{\hkbraket{k}}(Z_{k-1})$ and the digit permutation invariance of  $\sigma_a^{\hkbraket{k}}$ (\cite[Corollary 7.3]{kmrr1}), namely  
\begin{equation*}
    \phi\sigma^{\hkbraket{k}}_a = \sigma_a^{\hkbraket{k}} \quad \text{ for all digit permutation } \phi \colon X^{\hkbraket{k}} \to X^{\hkbraket{k}}.
\end{equation*}

\subsection{Furstenberg correspondence principle} \label{sec furstenberg correspondence}

In this section, We use the letter $I$ for a countable set of indices, where by \emph{countable} we include finite sets. We use a version of Furstenberg  correspondence principle which slightly generalize \cite[Proposition 1, Chapter 22]{Host_Kra_nilpotent_structures_ergodic_theory:2018} from finitely many to countably many sequences $\phi_i \in \ell^{\infty}(\N)$  (see also \cite[section 4]{Host_Kra09}).  The proof is basically identical by noticing that if $c_i = \norm{\phi_i}_{\infty}$ then the space $\prod_{i \in I}(c_i \D)$ for $\D$ the unit disc in $\C$, is a compact metric space even when $I$ is infinite.

\begin{proposition} \label{prop furstenberg correspondence 2}

Let $\phi_i \in \ell^{\infty}(\N)$ be a countable family of bounded sequences with $\norm{\phi_i}_{\infty} \leq C$ for some $C >0$ and every $i \in I$. Let $k \in \N \cup \{ \infty\}$ and suppose that  $\phi_i$ admits correlations along a F\o lner sequence $\Phi$ and that $\norm{\phi_i}_{U^k(\Phi)} =0$ for every $i \in I$. 
Then there exist an ergodic system $(X, \mu,T) $, a F\o lner sequence $\Psi $, a point $a \in \gen(\mu, \Psi)$ and continuous functions $f_i \in C(X)$ such that $\norm{f_i}_{U^k(X, \mu,T)} = 0 $ and
    \begin{align*}
        f_i(T^n a ) = \phi_i(n)  \quad \text{ for all } n \in \N, i \in I.
    \end{align*}
\end{proposition}

In the previous proposition, when $k = \infty$ we are using the abuse of notation $\norm{\phi_i}_{U^\infty(\Phi)} =0$ to denote $\norm{\phi_i}_{U^s(\Phi)} =0$ for all $s \in N$ (similarly for $\norm{f_i}_{U^\infty(X, \mu,T)} = 0 $). This case is only relevant for \Cref{sec infinite uniform sets}.

\begin{remark} \label{remark clopen for correspondence}
    In Proposition  \ref{prop furstenberg correspondence 2}, every sequence $\phi_i \in \ell^{\infty}(\N)$ has a continuous function $f_i \colon X \to \C$ that \emph{corresponds to it}. 
    Fixing $i\in I$, we can say more things about this function $f_i \in C(X)$. For instance if $R_i = \phi_i(\N) = \{ z \in \C \mid \phi_i(n)=z$ for some $n \in \N\}$ then the image $f_i(X)=\overline{R_i} \cup \{0\}$. 
In particular, if $\phi_i \in \ell^{\infty}(\N)$ takes finitely many values, $R_i=\{ r_1, \ldots, r_s\}$, then $f_i$ takes finitely many values and the set $E_j= f_i^{-1} (\{r_j\})$ is clopen for all $j=1, \ldots, s$. In that case, $f_i = \sum_{j=1}^{s} r_j \cdot \1_{E_j}$. We deduce the following case of interest. 
\end{remark}

\begin{corollary} \label{furstenberg correspondence for sets}
    Let $A_i \subset \N$ be $U^k(\Phi)$-uniform sets for $i \in I$ and $k \in \N \cup \{ \infty\}$. There exist  an ergodic system $(X, \mu,T) $, a F\o lner sequence $\Psi = (\Psi_N)_{N \in \N}$, a point $a \in \gen(\mu, \Psi)$ and $U^k(X, \mu,T)$-uniform clopen sets $E_i$ such that $\mu(E_i) = \diff_\Phi (A_i) >0$ and
    \begin{equation} \label{eq correspondence sequences}
        T^n a \in E_i \iff n \in A_i 
    \end{equation}
    for all $i\in I$, $n \in \N$.
\end{corollary}

\begin{proof}
    Direct from \cref{prop furstenberg correspondence 2} with $\phi_i = \1_{A_i} - \diff_\Phi(A_i)$ for all $i \in I$ and the properties discussed in \cref{remark clopen for correspondence}.
\end{proof}
        
\section{Continuous disintegration} \label{sec continuous disintegration}

We introduce the concept of continuous disintegration. This is a generalization of the continuous ergodic decomposition introduced in \cite{kmrr1}. This ensures compatibility between measurable and topological properties. 

\subsection{Definitions and properties}

\begin{definition} \label{defi cont desintegration}
  Let $(X, \mu,T)$ be a system and $(Y, \nu, S)$ be a factor. We say that the map $y \mapsto \mu_y$ from $Y \to \cM(X)$ is a disintegration of $\mu$ over $\nu$ if 
  \begin{itemize}
      \item $y \mapsto \mu_y (E)$ is measurable for every Borel $E \subset X$
      \item For every $f \in L^\infty(\mu)$
      \begin{equation*}
          \int_X f \diff \mu = \int_Y \int_X f \diff \mu_y \diff \nu(y).
      \end{equation*}
  \end{itemize}
 If the factor map $\pi \colon X \to Y$ is continuous and $y \mapsto \mu_y $ is continuous, then we say that it is a continuous disintegration over $\nu$. Taking $x \mapsto \mu_{\pi(x)}$ we can also think the continuous disintegration as a map from $X$ to $\cM(X)$ that is constant in the fibers of the map $\pi$. 
\end{definition}

In what follows we prove that the continuous disintegration respects natural relations between dynamical systems. 

\begin{proposition} \label{prop cont disintegration and conditional expectation}
    A system $\Xmt$ admits a continuous disintegration over a factor $(Y,\nu,T)$ with continuous factor map $\pi \colon X \to Y$ if and only if for every $f \in C(X)$ the conditional expectation $\E(f \mid Y)$ agrees $\nu$-almost surely with a continuous function. 
\end{proposition}

The proof is a direct generalization of \cite[Lemma 6.2]{kmrr1}. 

\begin{corollary} \label{cor concatenation of cont disintegration}
    Suppose that a measure preserving system $(X, \mu, T)$ admits a continuous disintegration over a factor $(Y,\nu,T)$ and also $(Y,\nu,T)$ admits a continuous disintegration over a factor $(Z,\lambda,T)$. Then $(X, \mu, T)$ admits a continuous disintegration over $(Z,\lambda,T)$. 
\end{corollary}

\begin{proof}
    By \cref{prop cont disintegration and conditional expectation}, for every $f \in C(X)$, $\E(f \mid Y)$ agrees  $\nu$-almost surely with a continuous function, let $g \in C(Y)$ be such that $g  = \E(f \mid Y)$. Thus, using again \cref{prop cont disintegration and conditional expectation}, we get that $\E( g \mid Z)$ agrees with a continuous function $\lambda$-almost surely. Thus, since $\lambda$-almost surely $\E(f \mid Z) = \E( \E(f \mid Y) \mid Z)$ and $g  = \E(f \mid Y)$, then $\E(f \mid Z) = \E( g \mid Z)$  which agrees, $\lambda$-almost surely, with a continuous function concluding the proof. 
\end{proof}

\begin{proposition} \label{prop inverse limit cont decomp}
     Suppose that for each $j \in \N$ we have a system $(X_j, \mu_j, T)$ that admits a continuous disintegration over a factor $(Y_j,\nu_j,T)$, suppose there are continuous factor maps $p_j \colon X_{j+1} \to X_j$ and $q_j \colon Y_{j+1} \to Y_j$. Then the inverse limit $(X, \mu,T)$ admits a continuous disintegration over the inverse limit $(Y,\nu,T)$. 
\end{proposition}

The proof of this proposition is also a direct generalization of \cite[Lemma 6.9]{kmrr1}.

\begin{lemma} \label{lemma continuity for neighborhoods}
    Let $\Xmt$ be a system that admits a continuous disintegration $y \mapsto \mu_y$ over a factor system $(Y,\nu,T)$. Suppose that $W \subset X$ is an open set such that $\mu_{y_0}(W)>0$ for some $y_0 \in Y$, then there exists a neighborhood $V\subset Y$ of $y_0$ such that for all $y \in V$, $\mu_{y}(W)>0$. 
\end{lemma}

\begin{proof}
    Take $W \subset X$ and $y_0 \in Y$ as in the statement, and consider an open set $W' \subset X$ with $\overline{W'} \subset W$ and $\mu_{y_0}(W') >0$. Consider $f \colon X \to [0,1]$ a continuous function for which $f(x) = 1$ for $ x \in W'$ and $f(x)=0$ for $x \in X \setminus W $. By construction $\int f \diff \mu_{y_0} >0$. Notice that $y \mapsto \int f \diff \mu_y$ is continuous and therefore there exists a neighborhood $V$ of $y_0$ such that $\int f \diff \mu_y >0$ for all $y \in V$. Finally, by construction, $0<\int f \diff \mu_y \leq \mu_y(W)$ for all $y\in Y$. 
\end{proof}

\begin{proposition} \label{prop the support is contained}
    Let $\Xmt$ be a system that admits a continuous disintegration $y \mapsto \mu_y$ over a factor system $(Y,\nu,T)$. For every $y \in \supp \nu$, $\supp \mu_y \subset \pi^{-1}(y)$. 
\end{proposition}

\begin{proof}

    Suppose there exists $y_0 \in \supp \nu$ and $x_0 \in \supp \mu_{y_0}$ such that $\pi(x_0) \neq y_0$ and consider $U,V \subset Y$ neighborhoods of $\pi(x_0)$ and $y_0$ respectively such that $U \cap V = \emptyset$. Let $\tilde U = \pi^{-1}(U)$, then since $x_0 \in \tilde U$, $\mu_{y_0}(\tilde U) >0$. Using that $y_0 \in \supp \nu$, by \cref{lemma continuity for neighborhoods}, we get that
    \begin{equation*}
        \nu( \{ y \in V \mid \mu_{y}(\tilde U) >0 \})>0.
    \end{equation*}
    In particular, since $\tilde U \cap \pi^{-1}(y) = \emptyset$,  
    \begin{equation*}
        \nu( \{ y \in Y \mid \mu_{y} (\pi^{-1}(y)) < 1 \})>0,
    \end{equation*}
    which contradicts \cite[Theorem 3.3 (ii)]{kmrr_BB} (see also \cite[Theorem 5.14]{Einsiedler_Ward11}).
\end{proof}

One can require something stronger than \cref{prop the support is contained}, we introduce the fully supported continuous disintegration. This is a key property for the sumset characterization of neighborhoods in pronilsystems.  

\begin{definition}
    Let $\pi \colon (X, \mu,T) \to (Y, \nu,S)$ be a continuous factor map with continuous disintegration $(\mu_y)_{y \in Y}$. 
    A continuous disintegration $(\mu_y)_{y \in Y}$ is \emph{fully supported} if for all $y \in Y$, $\supp \mu_y = \pi^{-1}(y)$.
\end{definition}

Like for the continuous disintegration, the fully supported continuous disintegration also respects many natural relations between dynamical systems. 

\begin{proposition}
    Let $\pi \colon (X, \mu,T) \to (Y, \nu,S)$ and $\theta \colon (Y, \nu,S) \to (Z, \rho, R)$ be continuous factor maps with fully supported continuous disintegration $(\mu_y)_{y \in Y}$ and $(\nu_z)_{z \in Z}$. Then the disintegration $(\tilde \mu_z)_{z \in Z}$ given by
    \begin{equation*}
        \tilde \mu_z = \int \mu_y \diff \nu_z(y)
    \end{equation*}
    is continuous and fully supported. 
\end{proposition}

\begin{proof}
    Pick $z \in Z$ and $x \in (\theta \circ \pi )^{-1}(z)$ arbitrary. Consider $f \colon X \to [0, \infty)$ continuous with $f(x)>0$. Since $(\mu_y)_{y \in Y}$ is fully supported, $x \in \supp \mu_{\pi(x)}$ and then  $\int f \diff \mu_{\pi(x)}  >0 $. Moreover, the map $y \mapsto \int f \diff \mu_y $ is continuous and therefore there exists a neighborhood $V$ of $\pi(x)$ such that $\int f \diff \mu_{y}  >0 $ for all $y \in V$.

    Now suppose, towards contradiction, that $\int f \diff \tilde \mu_z =0$ for $z = (\theta \circ \pi )(x)$, that is $\int \int f(x) \diff \mu_y (x)  \diff \nu_z(y) =0$ then 
    \begin{equation*}
        \nu_z\Big( \Big\{ y \mid \int f \diff \mu_y =0 \Big\} \Big) = 1.
    \end{equation*}
    In particular $\nu_z(V) =0$, but $\pi(x) \in \theta^{-1}(z) = \supp \nu_z$ and $V$ is a neigborhood of $\pi(x)$ which is a contradiction with the fact that $(\nu_z)_{z\in Z}$ is fully supported. 
   \end{proof}

\begin{lemma} \label{lemma inverse limit cont decomp fully supported}
     Suppose that for each $j \in \N$ we have a system $(X_j, \mu_j, T)$ that admits a fully supported continuous disintegration over a factor $(Y_j,\nu_j,T)$. Suppose there are continuous factor maps $p_j \colon X_{j} \to X_{j-1}$ and $q_j \colon Y_{j} \to Y_{j-1}$. Then the inverse limit $(X, \mu,T)$ admits a fully supported continuous disintegration over the inverse limit $(Y,\nu,T)$. 
\end{lemma}

\begin{proof}
     Pick $y \in Y$ and $x \in \pi^{-1}(y)$ arbitrary, we need to prove that $x \in \supp \mu_y$.
     
     Consider a real-valued non-negative continuous function $f \in C(X_j)$ with $f( \pi_j(x))>0$, where $\pi_j \colon X \to X_j$ is the continuous factor map such that $p_j \circ \pi_j = \pi_{j-1}  $. Similarly $\theta_j \colon Y \to Y_j$ is the continuous factor map such that $q_j \circ \theta_j = \theta_{j-1}$. Now
     \begin{align*}
         \int_X f \circ \pi_j \diff \mu_y = \E( f \circ \pi_j \mid Y) (y) = \E( f \mid Y_j) \circ \theta_j(y) = \int_{X_j} f \diff \mu_{j,\theta_j(y)}
     \end{align*}
     where $(\mu_{j,y})_{y \in Y_j}$ is the fully supported continuous disintegration of $\mu_j$. Now if $\psi_j \colon X_j \to Y_j$ is the factor map, then $\pi_j(x) \in \psi^{-1}_j(\theta_j(y))$, hence since $f( \pi_j(x))>0$ and $(\mu_{j,y})_{y \in Y_j}$ is fully supported $\int_{X_j} f \diff \mu_{j,\theta_j(y)} >0 $. Therefore $ \int_X (f \circ \pi_j) \diff \mu_y >0$. 

     With this, if $V$ is of the form $\pi_j^{-1}(V')$ for some neighborhood $V' \subset X_j$ of $\pi_j(x)$, then $\mu_y(V) >0$.  Finally, since $\bigcup_{j \in \N} \{ \pi_j^{-1}(V') \mid V' \subset X_j$ open$\}$ is a basis for the topology of $X$ we get that $x \in \supp \mu_y$. 
\end{proof}

The property of having a fully supported continuous disintegration is restrictive, with the following proposition we can deduce a variety of extensions that connot support such disintegration. 

\begin{proposition} \label{prop cont fully suppoerted disintegration implies open map}
    Let  $ (X, \mu,T)$ and $(Y, \nu,S)$ be dynamical system with $\supp \nu = Y$ and $\pi \colon (X, \mu,T) \to (Y, \nu,S)$ be a continuous factor maps with fully supported continuous disintegration $(\mu_y)_{y \in Y}$.  The factor map $\pi \colon X \to Y$ is open. 
\end{proposition}

\begin{proof}
    Consider an arbitrary open set $W \subset X$ and $y \in \pi(W)$. In particular, there exists $x \in \pi^{-1}(y) \cap W$, and by fully supported disintegration $\mu_y(W) >0$. By \cref{lemma continuity for neighborhoods}, since $y \in \supp \nu$, there exists an open neighborhood $V$ of $y$ such that for all $y' \in V$, $\mu_{y'}(W) >0 $. Now, $V \subset \pi(W)$, otherwise take $z \in V \setminus \pi(W)$, then $\pi^{-1}(z) \cap W = \emptyset$  but by definition $\mu_z(\pi^{-1}(z)) = 1$ which contradicts the fact that $\mu_z(W) >0$. With this we proved that $\pi(W)$ is open for every $W \subset X$ open.  
\end{proof}

\subsection{Continuous disintegration for pronilsystems} We now focus on the $s$-step nilsystem case. For $G$ a $s$-step nilpotent Lie group, we denote by $G_i$ the $i$\emph{th iterated commutator of} $G$, that is $G_1 = G$ and $G_i = [G_{i-1}, G] $ and when $X= G / \Gamma$, $e_X \in X$ is the image of the identity $e_G \in G$ under the quotient map. For every $1 \leq i \leq s$, $G_i$ is a rational normal subgroup where we say that a subgroup $H \leq G$ is \emph{rational} if the subset $H \cdot e_X$ is closed in $X$. 
If $H$ is a rational normal subgroup of $G$ we have that 
\begin{equation*}
    H \backslash X = G/(H\Gamma) = \frac{G/H}{(H\Gamma) /H}
\end{equation*}
thus the continuous group homomorphism given by the quotient map $q_H \colon G \to G/H$ induce a factor map between $(X, m_X,T)$ and the nilsystem $(H \backslash X, m_{H \backslash X}, T)$ where $T \colon \frac{G/H}{(H\Gamma) /H}\to \frac{G/H}{(H\Gamma) /H}$ is given by the translation by the element $q_H(\tau)$, where $\tau \in G$ is the element that defines the translation $T\colon X \to X$. In this case we have continuous disintegration,

\begin{proposition} \cite[Proposition 16, Chapter 10]{Host_Kra_nilpotent_structures_ergodic_theory:2018} \label{prop cont between normal subgroup}
    Let $X = G/\Gamma$ be a nilmanifold with Haar measure $m_X$ and $H$ be a normal rational subgroup of $G$. For each $x \in X$, $Y_x = H \cdot x$ is a nilmanifold with Haar measure $\nu_x$. Then the map $x \mapsto \nu_x$ defines a continuous disintegration of $m_X$ over the Haar measure of $H \backslash X$.
\end{proposition}

Notice that in the previous proposition, since $\nu_x$ is the Haar measure of $H \cdot x $ we also get that the disintegration is fully supported.  

\begin{corollary} \label{prop cont k-step max pronilfactors}
    Let $(X, m_X,T)$ be an ergodic $s$-step pronilsystem, then for every $k \leq s$, $(X, m_X,T)$ admits a fully supported continuous disintegration over $Z_{k}(X)$.
\end{corollary}

\begin{proof}
    By \cite[Chapter 13]{Host_Kra_nilpotent_structures_ergodic_theory:2018}, if $(X, m_X,T)$ is an ergodic $s$-step nilsystem,  $Z_k = G_{k+1} \backslash X$ and since $G_{k+1}$ is a rational normal subgroup of $G$ we conclude with \cref{prop cont between normal subgroup}. For the general case when $X$ is an inverse limit of $s$-step pronilsystem, we conclude by Propositions \ref{prop inverse limit cont decomp} and \ref{lemma inverse limit cont decomp fully supported}. 
\end{proof}

\subsection{Bounded marginals and lift of measures} \label{sec useful lemmas} 

For a probability space $(X, \cX, \mu)$ and a measure $\tau \in \cM(X^s)$ for $s\geq 1$, we say that $\tau$ \emph{has marginals bounded by} $\mu$, if for all $i = 1, \ldots, s$ there exists $c_i >0$ such that $\tau_i \leq c_i \mu$, that is
    \begin{equation*}
        \tau_i(A) \leq c_i \mu(A) \ \text{ for all } A \in \cX
    \end{equation*}
    where $\tau_i \in \cM(X)$ is the $i$th marginal of $\tau$.  
     If $(X, \mu, T)$ is a measure preserving system and $(Y , \nu, S)$ is a factor, with factor map $\pi \colon X \to Y$ and $\xi \in \cM(Y^s)$ is a measure with marginals bounded by $\nu$, then we say that the measure $\tau \in \cM(X^s)$ is \emph{the lift of} $\xi$ if for all $f_1, \ldots , f_s \in L^{\infty}(\mu)$, 
    \begin{equation} \label{def lift of a measure}
        \int_{X^s} f_1 \otimes \cdots \otimes f_s \diff \tau = \int_{Y^s} \E(f_1 \mid Y) \otimes \cdots \otimes \E(f_s \mid Y ) \diff \xi.  
    \end{equation}

\begin{lemma}  \label{technical lemma for positivity open maps}
    Let $(X, \mu, T)$ a measure preserving system and $(Y, \nu, S)$ a factor with continuous factor map $\pi \colon X \to Y$ and fully supported continuous disintegration. 
    
    Let $s \in \N$ and $\xi \in \cM(Y^s)$ be a measure with marginals bounded by $\nu$ and let $\tau$ be the lift of $\xi$. 
    Suppose $E_1, \ldots, E_s \subset X$ are open sets
    with positive measure and denote $\tilde E_i = \pi(E_i) \subset Y$ for $i = 1 \ldots, s$. If  $\xi ( \tilde E_1 \times \cdots \times \tilde E_s ) >0 $, then 
    \begin{equation*}
        \tau( E_1 \times \cdots \times E_s) >0. 
    \end{equation*}
\end{lemma}

\begin{proof}
    Let $E_1, \ldots, E_s \subset X$ and $\tilde E_1, \ldots, \tilde E_s \subset Y$ as in the statement. Notice that, by \cref{prop cont fully suppoerted disintegration implies open map}, the sets $\tilde E_i$ are open, in particular $\xi ( \tilde E_1 \times \cdots \times \tilde E_s )$ is well defined. Now,
    \begin{align*}
        \tau(E_1 \times \cdots \times E_s)  &=\int_{Y^s}  \E(\1_{E_1} \mid Y)\otimes \cdots \otimes  \E(\1_{E_s} \mid Y) \diff  \xi \\
        &\geq \int_{\tilde E_1 \times \cdots \times \tilde E_s}  \E(\1_{E_1} \mid Y)\otimes \cdots \otimes  \E(\1_{E_s} \mid Y) \diff  \xi. 
    \end{align*}
    Define, $F= \E(\1_{E_1} \mid Y)\otimes \cdots \otimes  \E(\1_{E_S} \mid Y) $. Notice that, $F$ is positive for $\xi$-almost every $\boldsymbol y \in \tilde E_1 \times \cdots \times \tilde E_s$, indeed, 
    \begin{align*}
        \xi( \{ \boldsymbol y \in \tilde E_1 \times \cdots \times \tilde E_s \mid F(y) =0\}) = \sum_{i=1}^k  \xi (\{ \boldsymbol y \in \tilde E_1 \times \cdots \times \tilde E_s \mid \E(\1_{E_i} \mid Y) (y_i) =0 \}) \\
        \leq \sum_{i=1}^k c_i \cdot \nu(\{ y_i \in \tilde E_i \mid \E(\1_{E_i} \mid Y) (y_i) =0 \}) \\
 = \sum_{i=1}^k c_i \cdot \mu(\{ x \in \pi^{-1}(\tilde E_i) \mid \mu_{\pi(x)}(E_i) =0 \}) =0.
    \end{align*}
    where the last equality is given by the fact that if $x \in \pi^{-1}(\tilde E_i)$ then there exist $x' \in E_i$ for which $\pi(x)=\pi(x')$ and then since the disintegration is fully supported, $0 <\mu_{\pi(x')}(E_i) = \mu_{\pi(x)}(E_i)$. 
    Finally,
    \begin{align*}
         \tau(E_1 \times \cdots \times E_s) 
        &\geq \int_{\tilde E_1 \times \cdots \times \tilde E_s}  F \diff  \xi  \\
        &\geq \sum_{n =0}^{\infty} \frac{1}{n+1}  \xi\bigg( \Big\{ y \in \tilde E_1 \times \cdots \times \tilde E_s \mid \frac{1}{n+1} \leq F(y) < \frac{1}{n} \Big\}\bigg) >0 
    \end{align*}
    where the positivity of the right hand side is ensured by the positivity of $F$ for almost every point in $\tilde E_1 \times \cdots \times \tilde E_s$.
\end{proof}

\section{Sumsets and return times on nilsystem} \label{sec nilsystem return times}

In this section, we start with general consequences for systems with fully supported continuous disintegration over the pronilfactors and then we apply those results to study topological properties of pronilsystems. 

For a system $(X,T)$, a point $x_0 \in X$ and a set $U \subset X$ we denote the set of return times $N(x_0,U) = \{  n \in \N \mid T^n x_0 \in U \} $. We also introduce the following notation, for a set $X$, a function $f \colon X \to \C$ and $i = 1, \ldots,k$, the function $f^{[i]} \colon X^k \to \C$ is  given by $(x_1, \ldots, x_k) \mapsto f(x_i) $. In particular, for $k =2$ and $f,g \colon X \to \C$, $f^{[1]} \cdot g ^{[2]} = f \otimes g$. 

\subsection{Sumsets in return time sets.}

The main goal of this section is to prove the following theorem.

\begin{theorem} \label{cor combinatorial BB and BCD points} \label{theorem main dynamical very important}
     Let $(X,\mu,T)$ be an ergodic system with topological pronilfactors, $a \in \gen(\mu, \Phi)$, for some F\o lner sequence $\Phi$  and $k \geq 2$. Let $\tilde X = \supp \mu$ and let $x_1, \ldots, x_k \in \tilde X$ with $\pi_{k-1}(a) = \pi_{k-1}(x_i)$ and $V_i$ open neighborhood of $x_i$ for $i=1, \ldots, k$. 
     \begin{enumerate}
         \item If $(\tilde X , \mu, T)$ admits a fully supported continuous disintegration over $Z_{k-1}(\mu)$, then there exists an infinite set $B \subset \N$ such that for all $i=1, \ldots, k$
         \begin{equation*}
             B^{\oplus i} \subset N(a,V_i),
         \end{equation*}
         \item If for $\ell \geq k$, $(\tilde X , \mu, T)$ admits a fully supported continuous disintegration over $Z_{\ell-1}(\mu)$, then $ 1 \leq \ell_1 < \ell_2 < \cdots < \ell_k \leq \ell$, there exist infinite sets $B_1 , \ldots, B_\ell \subset \N$ such that for all $i=1, \ldots, k$
         \begin{equation*}
             \sum_{j=1}^\ell \epsilon_j B_j \subset N(a,V_i) \quad \text{ for } \epsilon \in \hkbraket{\ell} \text{ with } |\epsilon| = \ell_i.
         \end{equation*}
     \end{enumerate}
\end{theorem}

To prove \cref{cor combinatorial BB and BCD points} we start with the following lemma. 

\begin{lemma}  \label{theorem mu almost every BB}
    Let $(X, \mu, T)$ be an ergodic system with topological pronilfactors and $k \geq 2$. Let $\tilde X = \supp \mu$ and suppose $(\tilde X , \mu, T)$ admits a fully supported continuous disintegration over $Z_{k-1}(\mu)$.
    
    For every $1 \leq \ell_1 < \cdots < \ell_k \leq \ell$ and for $\mu$-almost every $x_0 \in X$, if $x_1, \ldots, x_k \in \tilde X$ such that $\pi_{k-1}(x_0) = \pi_{k-1}(x_i)$  and $V_i$ is a neighborhood of $x_i$ for $i=1, \ldots, k$, then there exists an Erd\H{o}s progression $(x_0, y_1, \ldots, y_\ell) \in X^{\ell+1}$ such that $y_{\ell_i} \in V_i$ for all $i = 1, \ldots, k$. 
\end{lemma}

\begin{proof}
    Fix $1 \leq \ell_1 < \cdots < \ell_k \leq \ell$ and for any point $x \in X$, and let $\xi_x \in \cM(Z_{\ell-1}^\ell)$ and $\sigma_x \in \cM(X^{\ell+1})$ be measures introduced in \cref{definition of sigma}. Likewise define
    \begin{equation*}
        N_{\ell_1, \ldots, \ell_k} (x) = \overline{\{ (T^{\ell_1} \times \cdots \times T^{\ell_k})^n (\pi_{k-1}(x), \ldots, \pi_{k-1}(x))  \mid n \in \Z\} } \subset Z_{k-1}^k 
    \end{equation*}
    and $\tilde \xi_x$ its unique $(T^{\ell_1} \times \cdots \times T^{\ell_k})$-invariant measure. 
    We highlight that $\xi_x \in \cM(Z_{\ell-1}^{\ell})$ and $\tilde \xi_x \in \cM(Z_{k-1}^{k})$. It can be derived from \cite{ziegler2005nonconventional} (see also \cite{Bergelson_Host_Kra05}) that for $\mu$-almost every $x \in X$, if $f_1, \ldots, f_k \in C(X)$ then
    \begin{equation} \label{eq equality that happends mu almost surely}
        \int_{X^{\ell+1}} \prod_{i=1}^k f_i^{[\ell_i]} \diff \sigma_x = \int_{Z^{k}_{k-1}} \bigotimes_{i=1}^k \E(f_i \mid Z_{k-1}) \diff \tilde \xi_x. 
    \end{equation}

    Pick $x_0 \in X$ in the set of full measure for which \eqref{eq equality that happends mu almost surely} holds true and for which $x_0 \in \gen(\mu, \Phi)$ for some F\o lner sequence. By unique ergodicity and minimality of $N_{\ell_1, \ldots, \ell_k} (x)$, for every neighborhood $\tilde U$ of $\pi_{k-1}(x_0)$
    \begin{equation} \label{eq positivity nbh}
        \tilde \xi_{x_0} ( \tilde U \times \cdots \times \tilde U) >0. 
    \end{equation}
    Consider $V_i$ an aribitrary neighborhood of $x_i$ and $\tilde V = \pi_{k-1}(V_1) \cap \cdots \cap \pi_{k-1}(V_k)$. Notice that, by \cref{prop cont fully suppoerted disintegration implies open map}, $\tilde V$ is an open subset of $Z_{k-1}$ and $\pi_{k-1}(x_0) \in \tilde V$, hence by \eqref{eq positivity nbh}, $\tilde \xi_{x_0}( \tilde V_1 \times \cdots \tilde V_k) \geq \tilde \xi_{x_0}( \tilde V \times \cdots \tilde V) >0$. Thus, since $\Xmt$ admits a fully supported continuous disintegration over $Z_{k-1}(\mu)$, by  \cref{technical lemma for positivity open maps} and \eqref{eq equality that happends mu almost surely}, we get that
    \begin{align*}
        \sigma_{x_0} \Big( \bigcap_{i=1}^k \{ (x_0, y_1, \ldots, y_\ell) \mid y_{\ell_i} \in V_i \} \Big) = \int_{X^{\ell+1}} \prod_{i=1}^k \1_{V_i} ^{[\ell_i]} \diff \sigma_{x_0} \\ \geq \int_{(\supp \mu)^{\ell+1}} \prod_{i=1}^k \1_{V_i} ^{[\ell_i]} \diff \sigma_{x_0} >0.
    \end{align*}
     With that, using \cref{lemma characterization progressive measures} (i), we conclude that for every neighborhoods $V_i$ of $x_i$ there is an Erd\H{o}s progression $(x_0,y_1, \ldots, y_\ell)$ such that $y_{\ell_i} \in V_i$. 
\end{proof}

In \Cref{sec open q}, we discuss why the conclusion of \cref{theorem mu almost every BB} is only true in general for $\mu$-almost every $x\in X$. However, when $k= \ell$ we can change the $\mu$-almost surely quantifier by a for all, namely

\begin{proposition} \label{theorem BB points}
     Let $(X, \mu, T)$ be an ergodic with topological pronilfactors, $\Phi$ a F\o lner sequence, $a \in \gen(\mu,\Phi)$ and $k \geq 2$. Let $\tilde X = \supp \mu$ and suppose $(\tilde X , \mu, T)$ admits a fully supported continuous disintegration over $Z_{k-1}(\mu)$. If $x_1, \ldots, x_k \in \tilde X$ such that $\pi_{k-1}(a) = \pi_{k-1}(x_i)$  and $V_i$ is a neighborhood of $x_i$ for $i=1, \ldots, k$, then there exists an Erd\H{o}s progression $(a, y_1, \ldots, y_k) \in X^{k+1}$ such that $y_{i} \in V_i$ for all $i = 1, \ldots, k$. 
\end{proposition}

\begin{proof}
    In this case \eqref{eq equality that happends mu almost surely} happens for every generic point $a \in \gen(\mu, \Phi)$, because in this case, the measure $\xi_a$ from \cref{definition of sigma} and the measure $\tilde \xi_a$ in the proof of \Cref{theorem mu almost every BB} coincide. 
\end{proof}

We derive the following theorem for the minimal case. 

\begin{theorem}  \label{thrm minaml same image in Z_{k-1}}
    Let $(X,\mu,T)$ be a minimal system with topological pronilfactors, $ k \geq 2$ and $\Xmt$ admits a fully supported continuous disintegration over $Z_{k-1}(\mu)$. Let $x_0, \ldots, x_k \in X$ with $\pi_{k-1}(x_0) = \pi_{k-1}(x_i)$ for $i=1, \ldots, k$. Then for every $1 \leq \ell_1 < \cdots < \ell_k \leq \ell$ there exists $\boldsymbol{y} \in Q^{\llbracket \ell \rrbracket}(X)$ such that $y_{\underline{0}} = x_0$ and for all $i=1, \ldots, k$
    \begin{equation*}
        y_{\epsilon} = x_i \quad \text{ for all } \epsilon \in \hkbraket{\ell} \text{ with } |\epsilon| = \ell_i.
    \end{equation*}
\end{theorem}

\begin{proof} 

    We prove first the statement for a set of full measure. 
    
    Fixing $1 \leq \ell_1 < \ell_2 < \cdots < \ell_k \leq \ell$ suppose $x_0 \in X$ is in the set of full measure where \cref{theorem mu almost every BB} holds true. Then, for every $n \in \N$, we denote  $B(x_i, 1/n)$ be the ball of center $x_i$ and radius $1/n$, there exists an Erd\H{o}s progression $(x_0,y_1^{(n)}, \ldots, y_\ell^{(n)})$ such that $y_{\ell_i}^{(n)} \in B(x_i, 1/n)$ for all $i=1, \ldots, k$. 
    Consider $\psi_\ell \colon X^{\ell +1} \to X^{\hkbraket{\ell}}$ the $\ell$-vertex embedding, by \cref{lemma erdos progression vs erdos cube},  $\psi_\ell(x_0, y_1^{(n)}, \ldots, y_\ell^{(n)}) \in Q^{\hkbraket{\ell}}(X)$.
    Since $Q^{\hkbraket{\ell}}(X)$ is closed we conclude by taking limit of (possibly a subsequence of) $\psi_\ell(x_0, y_1^{(n)}, \ldots, y_\ell^{(n)})$. 
    This proves the statement for $\mu$-almost every $x_0 \in X$. 
    
    Now in general, take $x_0, \ldots, x_k \in X$ as in the statement. As before, for a given $n \in \N$, let $V_i = B(x_i, 1/n)$ be the ball of center $x_i$ and radius $1/n$, for all $i =0,1, \ldots, k$. Since $\pi_{k-1} \colon X \to Z_{k-1}$ open, then every $\pi_{k-1}(V_i) = \tilde V_i$ is a neigborhood of $\pi_{k-1}(x_0)$. Take $\tilde V = \tilde V_1 \cap \cdots \cap \tilde V_k$ and consider $V'_i = V_i \cap \pi_{k-1}^{-1}(\tilde V)$, then clearly $V'_i $ is an open neighborhood of $x_i$ contained in $V_i$. 
    By minimality $\supp \mu = X$,  so there exists $x \in V_0'$ such that \cref{theorem mu almost every BB} holds true. Also, by construction, for every $i=1, \ldots,k$ there exists $y_i \in V_i'$ such that $\pi_{k-1}(y_i) = \pi_{k-1}(x)$. Thus, for all $n \in \N$ and $i=1, \dots, k$, we find an Erd\H{o}s progression $ (z^{(n)}_0, \ldots, z_\ell^{(n)})$ such that $z_0^{(n)} \in B(x_0,1/n)$ and $z_{\ell_i}^{(n)} \in B(x_i, 1/n)$. We conclude again by taking limit of $\psi_{\ell}(z^{(n)}_0, \ldots, z_\ell^{(n)}) \in  Q^{\hkbraket{\ell}}(X)$. 
\end{proof}

\begin{proposition} \label{theorem BCD points}
     Let $(X,\mu,T)$ be an ergodic system with topological pronilfactors, $a \in \gen(\mu, \Phi)$ for some F\o lner sequence $\Phi$  and $2 \leq k \leq \ell$. Denote $\tilde X = \supp \mu$ and suppose $(\tilde X , \mu, T)$ admits a fully supported continuous disintegration over $Z_{\ell-1}(\mu)$. 
     
     If $x_1, \ldots, x_k \in \tilde X$ with $\pi_{k-1}(a) = \pi_{k-1}(x_i)$ and $V_i$ open neighborhood of $x_i$ for $i=1, \ldots, k$, then for all integers $1 \leq \ell_1 < \cdots < \ell_k \leq \ell$ there exists an Erd\H{o}s cube $\boldsymbol{y} \in E^{\llbracket \ell \rrbracket}_a(X)$ such that for all $i=1, \ldots, k$ 
    \begin{equation*}
       \boldsymbol y_{\epsilon} \in V_i \quad \text{ for all } \epsilon \in \hkbraket{\ell} \text{ with } |\epsilon| = \ell_i.
    \end{equation*} 
\end{proposition}

\begin{proof}

    We denote $\tilde a = \pi_{\ell-1}(a)$, $\tilde x_i = \pi_{\ell-1}(x_i)$ and $\theta_{k-1} \colon Z_{\ell-1} \to Z_{k-1}$ the natural factor map which fulfills $\pi_{k-1} = \theta_{k-1} \circ \pi_{\ell-1}$.
    
    By \cref{thrm minaml same image in Z_{k-1}} since $Z_{\ell-1}$ is minimal and it admits a fully supported continuous disintegration over $Z_{k-1}$ (see \cref{prop cont k-step max pronilfactors}) and $\theta_{k-1} (\tilde a) = \theta_{k-1}(\tilde x_i)$ for all $i = 1, \ldots, k$, there exists $\boldsymbol{\tilde y} \in Q^{\hkbraket{\ell}}(Z_{\ell-1})$ with $\tilde y_{\epsilon} = \tilde x_i$ for all $|\epsilon| = \ell_i$, $i = 1, \ldots , k$ and $\tilde y_{\underline{0}} = \tilde a$. 
    By minimality of $Q^{\hkbraket{\ell}}_{\tilde a}(Z_{\ell-1})$, we get that for any neighborhood $U$ of $\boldsymbol{\tilde y}$, $\xi^{\hkbraket{\ell}}_a (U) >0$.

     In particular if $\tilde V_i = \pi_{\ell-1}(V_i)$, since $\pi_{\ell-1}$ is open (see \cref{prop cont fully suppoerted disintegration implies open map}), then $U = \{ z \in Q^{\hkbraket{\ell}}(Z_{\ell-1}) \mid z_\epsilon \in \tilde V_i , |\epsilon| = \ell_i, i = 1, \ldots, k\} $ is an open neighborhood of $\boldsymbol{\tilde y}$ and therefore $\xi^{\hkbraket{\ell}}_a (U) >0$. Finally, from \cref{definition of sigma cubic} we get that
    \begin{equation*}
         \sigma_a^{\hkbraket{\ell}}\Big( \bigcap_{i=1}^k \Big( \bigcap_{\epsilon \in \hkbraket{\ell}, |\epsilon|= \ell_i} \{ \boldsymbol{y} \in X^{\hkbraket{\ell}} \mid y_{\epsilon} \in V_i \} \Big) \Big) = \int_{Z_{\ell-1}^{\hkbraket{\ell}}}  \prod_{|\epsilon| = \ell_1, \ldots, \ell_k} \E(\1_{V_i} \mid Z_{\ell-1})^{[\epsilon]} \diff \xi_a^{\hkbraket{\ell}}
    \end{equation*}
    which is positive by \cref{technical lemma for positivity open maps} and the fact that $(X, \mu,T)$ admits a fully supported continious disintegration over $Z_{\ell-1}$. We conclude with \cref{lemma sigma^[k] property} (ii). 
\end{proof}

We can now prove the main result. 

\begin{proof}[Proof of \cref{cor combinatorial BB and BCD points}]
     By \cref{theorem BB points}, under the assumptions of \cref{cor combinatorial BB and BCD points}, there exists an Erd\H{o}s progression $(a, y_1, \ldots, y_k) \in X^{k+1}$ such that $y_{i} \in V_i$ for all $i = 1, \ldots, k$. Since $V_i$ is open, using \cref{Erdos cubes and progressions and sumsets} we conclude the first claim of \cref{cor combinatorial BB and BCD points}. For the second one we first use \cref{theorem BCD points} to find the Erd\H{o}s cube such that  $i=1, \ldots, k$ 
    \begin{equation*}
       \boldsymbol y_{\epsilon} \in V_i \quad \text{ for all } \epsilon \in \hkbraket{\ell} \text{ with } |\epsilon| = \ell_i.
    \end{equation*} 
    and we conclude again using \cref{Erdos cubes and progressions and sumsets}. 
\end{proof}

\subsection{Nilsystem and topological dynamics} \label{sec RPk}
From \cite[Definition 3.2]{Host_Kra_Maass_nilstructure:2010}, let $(X,T)$ be a transitive system and $k \in \N$, we say that two points $x,y \in X$ are \emph{regionally proximal of order} $k$, denoted $(x,y) \in \RP^{[k]}(X)$, if for every $\delta >0$ there exist $x',y' \in X$ and $n \in \Z^k$ such that $d(x,x') < \delta$, $d(y,y') < \delta$ and 
\begin{equation*}
    d(T^{\epsilon \cdot n } x' , T^{\epsilon \cdot n} y') < \delta \text{ for all } \epsilon \in \hkbraket{k}^*.
\end{equation*}
By \cite{Shao_Ye_regionally_prox_orderd:2012} we know that $\RP^{[k]}(X)$ is an equivalence relation for minimal systems. By \cite[Theorem 1.2]{Host_Kra_Maass_nilstructure:2010}, $(X,T)$ is a $k$-step pronilsystem if and only if $\RP^{[k]}(X)$ is the identity relation (that is $(x,y) \in \RP^{[k]}(X)$ if and only if $x=y$). With that if $(X, \mu,T)$ is a minimal and ergodic system and one supposes that $\pi_{k} \colon X \to Z_{k}(\mu)$ is continuous, then it is clear that $Z_{k}(\mu) = X/\RP^{[k]}(X)$. In particular if $(X, \mu,T)$ is a minimal pronilsystem, then $Z_k(\mu) = X/\RP^{[k]}(X)$ for every $k \in \N$. 

The following theorem can be stated for any minimal system with continuous fully supported disintegration over their pronilfactors but we stick to the pronilsystem case.

\begin{theorem} \label{theorem summary} 
    Let $k \geq 2$ be an integer, $(X,T)$  be a minimal pronilsystem. For $x,y \in X$, the following are equivalent
    \begin{enumerate}
        \item \label{summary theorem pt 0} $(x,y) \in \RP^{[k-1]}(X)$;
        \item \label{summary theorem pt 1} For every neighborhood $V$ of $y$ there exists $b_1, \ldots, b_k \in \N$ distinct natural numbers such that 
        $$\sum_{i \in I} b_i \in N(x,V) \quad \text{ for all } \quad I \subset \{1,\ldots, k\}, \ I \neq \emptyset; $$
        \item \label{summary theorem pt 2} For every neighborhood $V$ of $y$ there exists an infinite set $B \subset \N$ such that $$B \cup \cdots \cup B^{\oplus k} \subset N(x,V);$$ 
        \item \label{summary theorem pt 3} For every neighborhood $V$ of $y$, there are infinite sets $B_1, \ldots, B_k \subset \N$ such that 
    \begin{equation*}
        \sum_{i=1}^k \epsilon_i B_i \subset N(x,V) \text{ for all } \ \epsilon \in \hkbraket{k}^*;
    \end{equation*}
    \item \label{summary theorem pt 4} For every neighborhood $V$ of $y$ and for all integers $1 \leq \ell_1 < \cdots < \ell_k <\ell$, there are infinite sets $B_1, \ldots, B_{\ell} \subset \N$ such that for all $i=1, \ldots, k$ 
    \begin{equation*}
        \sum_{j=1}^\ell \epsilon_j B_j \subset N(x,V) \ \text{ for all } \epsilon \in \hkbraket{\ell} \text{ with } |\epsilon |= \ell_i;
    \end{equation*} 
    \item \label{summary theorem pt 5} For every neighborhood $V$ of $y$, every neighborhood $U$ of $x$ and all integers $\ell_1 < \cdots < \ell_k$, there exists an infinite set $B \subset \N$ and a point $x' \in U$ such that $$B^{\oplus \ell_1} \cup \cdots \cup B^{\oplus \ell_k} \subset N(x',V).$$
\end{enumerate}
    
\end{theorem}

The first two items equivalence are true for any minimal system, this is a consequence of a characterization of $\RP^{[k-1]}(X)$ proved in \cite{Shao_Ye_regionally_prox_orderd:2012}. 
We first give another characterization of $\RP^{[k]}(X)$.

\begin{proposition} \label{prop top dynamics cube caracterization}
    Let $(X, \mu,T)$ be an ergodic pronilsystem. Let also $x,y \in X$ be points, the following are equivalents
    \begin{enumerate}
        \item $(x,y) \in \RP^{[k-1]}(X)$,
        \item $(x, y, \ldots, y ) \in Q^{\hkbraket{k}}(X)$
        \item For every, $1 \leq \ell_1< \ell_2 < \cdots < \ell_k \leq \ell$, There exists $\boldsymbol z\in Q^{\hkbraket{\ell}}(X)$ such that $z_{\underline{0}} = x$ and $z_\epsilon = y$ for all $\epsilon \in \hkbraket{\ell}$ with $|\epsilon| = \ell_1, \ldots, \ell_k $. 
    \end{enumerate}
\end{proposition}

\begin{proof}
    The first two equivalences are given in \cite[Theorem 3.4]{Shao_Ye_regionally_prox_orderd:2012} and only require minimality of $(X,T)$. It is clear that the third item implies the second one. The fact that $(x,y) \in \RP^{[k-1]}(X)$ implies the last item is given by \Cref{thrm minaml same image in Z_{k-1}} taking $y=y_1 =  \cdots = y_k$.
\end{proof}

\begin{proposition} \label{lemma equiv RP BCD}
    Let $(X,T)$ be a minimal system and let $x,y \in X$ be two points. If for all neighborhoods $V$ of $y$ and $U$ of $x$, an Erd\H{o}s cube $\boldsymbol z \in E^{\hkbraket{k+1}}(X)$ with $z_{\underline{0}} \in U$ and $z_\epsilon \in V$ for all $\epsilon \in \hkbraket{k+1}^*$, then $(x,y) \in \RP^{[k]}(X)$.
\end{proposition}

\begin{proof}
    We know that the set of Erd\H{o}s cubes $E^{\hkbraket{k+1}}(X)$ is contained in the set of dynamical cubes  $Q^{\hkbraket{k+1}}(X)$. Then the assumption implies that $(x, y, \ldots, y) \in \overline{E^{\hkbraket{k+1}}(X)} \subset Q^{\hkbraket{k+1}}(X)$, so we conclude by \cite[Theorem 3.4]{Shao_Ye_regionally_prox_orderd:2012}. 
\end{proof}


 Notice that in the previous proposition the only assumption on $(X,T)$ is minimality. 
We also have a version of \cref{lemma equiv RP BCD} with Erd\H{o}s progressions. 

\begin{proposition} \label{lemma equiv RP}
    Let $(X,T)$ be a minimal system and let $x,y \in X$ be two points. If for all neighborhoods $V$ of $y$ and $U$ of $x$, there exist $x' \in U$ and an Erd\H{o}s progression $(x', y_1, \ldots, y_k,y_{k+1})$ with $y_i \in V$ for all $i= 1, \ldots, k+1$, then $(x,y) \in \RP^{[k]}(X)$.
\end{proposition}

\begin{proof}
    Using the $(\ell+1)$-vertex embedding defined in \cref{def ell-vertex}, then for $U,V \subset X$ arbitrary neighborhoods as in the statement, there exists $(x', y_1, \ldots, y_k,y_{k+1}) \in X^{k+2}$ such that  $\psi_{\ell+1}(x', y_1, \ldots, y_k,y_{k+1})  \in U \times V^{\hkbraket{k+1}^*} $ is an Erd\H{o}s cube and so we conclude again by \cref{lemma equiv RP BCD}.
\end{proof}

With the previous propositions we prove \cref{theorem summary},

\begin{proof}[Proof of \Cref{theorem summary}]
    The equivalence of items \ref{summary theorem pt 0} and \ref{summary theorem pt 1} is already established. It is clear that \cref{summary theorem pt 4} implies \cref{summary theorem pt 3}. 
    By \Cref{cor combinatorial BB and BCD points}, \cref{summary theorem pt 0} implies \cref{summary theorem pt 2} and also \cref{summary theorem pt 0} implies \cref{summary theorem pt 4} where we use that in ergodic prosystem every point is generic along any F\o lner sequence for its unique invariant measure. Similarly, by \Cref{theorem mu almost every BB}, \cref{summary theorem pt 0} implies \cref{summary theorem pt 5}. 

    Consider \cref{summary theorem pt 5} for $\ell_1 = 1, \ldots, \ell_k = k$, then by \Cref{lemma equiv RP}, it implies \cref{summary theorem pt 0}. Similarly, using again \Cref{lemma equiv RP}, \cref{summary theorem pt 2} implies \cref{summary theorem pt 0}. Finally by \Cref{lemma equiv RP BCD}, \cref{summary theorem pt 3} implies \cref{summary theorem pt 0}.
\end{proof}







\section{Uniformity sets} \label{sec uniform sets}

In this section we prove a series of consequences for $U^k(\Phi)$-uniform sets, in particular \cref{theorem uniformity BB}. Using \cref{furstenberg correspondence for sets}, we translate the problem to a dynamical system setting and then we prove it. We only explicitly show how to deduce \cref{theorem uniformity BB} from its dynamic reformulation (that is \cref{theorem dynamic reformulation}), since the proof for the other cases are identical.

\subsection{Proof of Theorem \ref{theorem uniformity BB}} \label{sec dynamic reformulation} The following statement is equivalent to \cref{theorem uniformity BB}. 

\begin{theorem} \label{theorem dynamic reformulation}
    Let $(X,\mu,T)$ be an ergodic system with topological pronilfactors, $a \in \gen(\mu,\Phi)$ for some F\o lner sequence $\Phi$, $k \geq 2$ and $E_1, \ldots, E_k\subset X$ be $U^k(X,\mu,T)$-uniform clopen sets. Then there exists an infinite set $B \subset \N$ such that
         \begin{equation*}
             B^{\oplus i} \subset N(a,E_i) \quad \text{ for all } i =1, \ldots, k. 
         \end{equation*} 
\end{theorem}

\begin{proof}
    Let $E_1, \ldots, E_k\subset X$ as in the statement, in particular, $\E(\1_{E_i} \mid Z_{k-1}) = \mu(E_i)$ $\mu$-almost surely for all $i = 1, \ldots, k$. Then by definition of the measure $\sigma_a$ (see \cref{definition of sigma}) 
    \begin{equation*}
        \sigma_a(X \times E_1 \times \cdots \times E_k) = \prod_{i=1}^k \mu(E_i) >0.
    \end{equation*}
    The conclusion is a direct consequence of \cref{lemma characterization progressive measures}. 
\end{proof}

\begin{proof}[Proof that \cref{theorem dynamic reformulation} is equivalent to \cref{theorem uniformity BB}]

For the forward direction, it suffices to show that for $U^k(\Phi)$-uniform sets $A_1, \ldots, A_k \subset \N$, there exists systems and points as in the statement of \cref{theorem dynamic reformulation} such that $N(a,E_i ) = A_i $ for all $i=1, \ldots, k$.

For the sets $A_1, \ldots, A_k$ it suffices to take the ergodic system $(X, \mu,T)$, the generic point $a \in \gen(\mu,\Phi)$ and the clopen $U^k(X,\mu,T)$-uniform sets $E_1, \ldots, E_k$ given by \cref{furstenberg correspondence for sets}. Using \cite[Lemma 5.8]{kmrr1}, we can suppose that $(X, \mu,T)$ has topological pronilfactors. We conclude by noticing that \eqref{eq correspondence sequences} implies $A_i = N(a, E_i)$ for all $i=1, \ldots, k$.

For the converse, consider  $U^k(X,\mu,T)$-uniform clopen sets $E_1, \ldots, E_k \subset X$. 
We have that $\1_{E_i} \colon X \to \{0,1\}$ is a continuous function and if $A_i = N(a,E_i)$ then since $a \in \gen(\mu, \Phi)$

\begin{align*}
    \norm{\1_{A_i} - \mu(E_i)}_{U^k(\Phi)} &= \lim_{H \to \infty} \lim_{N \to \infty} \frac{1}{H^k} \sum_{h \in \{1, \ldots, H\}^k} \frac{1}{|\Phi_N|} \sum_{n \in \Phi_N} \prod_{\epsilon \in \hkbraket{k}} [\1_{A_i}(n + \epsilon \cdot h) - \mu(E_i) ] \\
    &= \lim_{H \to \infty} \lim_{N \to \infty} \frac{1}{H^k} \sum_{h \in \{1, \ldots, H\}^k}  \frac{1}{|\Phi_N|} \sum_{n \in \Phi_N} \prod_{\epsilon \in \hkbraket{k}} [\1_{E_i}(T^{n + \epsilon \cdot h} a) - \mu(E_i) ] \\
    &= \lim_{H \to \infty}  \frac{1}{H^k} \sum_{h \in \{1, \ldots, H\}^k} \int_X \prod_{\epsilon \in \hkbraket{k}} [T^{n + \epsilon \cdot h} \1_{E_i}- \mu(E_i)]  \diff \mu \\
    &= \norm{\1_{E_i}- \mu(E_i)}_{U^k(X, \mu,T)} = 0.
\end{align*}

Thus the sets $A_i$ are $U^k(\Phi)$-uniform and $\diff_\Phi(A_i) = \mu(E_i)$ for all $i = 1, \ldots, k$, therefore \cref{theorem uniformity BB} implies \cref{theorem dynamic reformulation}. 
\end{proof}

\subsection{Infinitely uniform sets} \label{sec infinite uniform sets} We give a $U^{\infty}(\Phi)$-uniform  version of \cref{theorem uniformity BB}.

\begin{theorem} \label{U^infty very strong main theorem}  
Let $\Phi = (\Phi_N)_{N \in \N}$ be a F\o lner sequence in $\N$. Suppose that $A_1,A_2, \ldots \subset \N$ are $U^{\infty}(\Phi)$-uniform sets. Then there exists a sequence of nested infinite subsets of the natural numbers $B_1 \supset B_2 \supset B_3 \supset \ldots$ such that 
    \begin{equation*}
        B_i^{\oplus i}\subset A_i  \quad \text{ for all } i \geq 1. 
    \end{equation*}
     
\end{theorem}

Its dynamical formulation is the following.

\begin{proposition}  \label{U^infty dynamical} 
    Let $(X,\mu,T)$ be an ergodic system with topological pronilfactors, $a \in \gen(\mu,\Phi)$ for some F\o lner sequence and $E_1, E_2, \ldots \subset X$ be $U^\infty(X,\mu,T)$-uniform clopen sets. Then there exists an infinite set $B = \{ b_1 < b_2 < \cdots \} \subset \N$ such that
         \begin{equation*}
             (B \setminus \{b_1, \ldots, b_{i-1}\})^{\oplus i} \subset N(a,E_i) \quad \text{ for all } i \in \N. 
         \end{equation*} 
\end{proposition}

\begin{proof} 
    Let $\xi \in \cM(Z_\infty^{\N}(\mu))$ be the unique invariant measure of 
    \begin{equation*}
        \overline{\{ (T \times T^2 \times \cdots )^n(\pi_{\infty}(a), \pi_{\infty}(a), \cdots) \in Z_\infty^{\N}(\mu) \mid n \in \Z\} }
    \end{equation*}
    where $\pi_{\infty} \colon X \to Z_\infty (\mu)$ is the continuous factor map. Let $\sigma \in \cM(X^{\N_0})$ be given by
    \begin{equation*}
        \sigma_{a} = \delta_{a} \times \tilde \sigma_{a} 
    \end{equation*}
    where $ \tilde \sigma_{a} $ is the lift of $\xi \in \cM(Z_\infty^{\N}(\mu))$. For $r \in \N$, $\sigma^{(r)}$ denotes the projection of $\sigma_{a}$ to the coordinates $\{0,\ldots,r\}$. From \cite[Proposition 5.7]{hernandez_kousek_radic2025density_Hindman} we have that, for $r \in \N$, if $W \subset X^r$ is an open set such that $\sigma^{(r)}(X \times W) >0$  then there exist infinitely many $n \in \N$ such that
    \begin{equation} \label{eq infty progressive.}
      \sigma^{(r)}((X \times W ) \cap (T \times \cdots \times T)^{-n}(W \times X)) >0.
    \end{equation}
    With this, since $E_i$ is $U^\infty(X,\mu,T)$-uniform, then $\E(\1_{E_i} \mid Z_{\infty}(\mu)) = \mu(E_i)$ for all $i \in \N$. In particular, $\sigma^{(1)}(X \times E_1) = \mu(E_1) >0$, so by \eqref{eq infty progressive.} there exists $b_1\in \N$, such that $\sigma^{(1)}((X \times E_1) \cap (T \times  T)^{-b_1}(E_1 \times X)) >0$. Likewise,
    \begin{equation*}
        \sigma^{(2)}(T^{-b_1}E_1 \times E_1 \times E_2) = \sigma^{(1)}(T^{-b_1}E_1 \times E_1 ) \cdot \mu(E_2)  >0,
    \end{equation*}
    so again by \eqref{eq infty progressive.} applied to $X \times E_1 \times E_2$ and noticing that the $0$th coordinate is a Dirac measure, there exists $b_2 > b_1$ such that
    \begin{equation*}
        \sigma^{(2)}((T^{-b_1}E_1 \times E_1 \times E_2) \cap (T \times T \times T)^{-b_2}(E_1 \times E_2 \times X)) >0.
    \end{equation*}
    Using the same reasoning with $E_3$, we get $b_3 > b_2$ such that
    \begin{align*}
        \sigma^{(3)}((T^{-b_1}E_1 \cap T^{-b_2}E_1) &  \times (E_1 \cap T^{-b_2} E_2) \times E_2 \times E_3)  \\
        \cap &(T \times \cdots \times T)^{-b_3}((E_1 \cap T^{-b_2} E_2) \times E_2 \times E_3 \times X)) >0.
    \end{align*}
    In particular, we get in the $0$th coordinate the set $S_3 = (T^{-b_1}E_1 \cap T^{-b_2}E_1 \cap T^{-b_3}E_1 \times T^{-(b_2 + b_3)} E_2) $.
    Repeating this procedure inductively, for every $r \in \N$, we find in the $0$th coordinate the following set
    \begin{equation*}
        S_{2r+1} = \bigg(\bigcap_{j=1}^{2r+1} T^{-b_j} E_1 \bigg) \cap \bigg(\bigcap_{2 \leq j_1 < j_2 \leq 2r+1} T^{-(b_{j_1} + b_{j_2})} E_2 \bigg) \cap \cdots \cap (T^{b_{r+1} +\cdots + b_{2r+1} } E_r).
    \end{equation*}
    By definition of $\sigma_a$, $a \in S_{i}$ for all $i \in \N$, so we conclude by taking $a \in \bigcap_{r \in \N} S_{2r+1}$. 
    \end{proof}

    Notice that, since $E_i$ are clopen subsets for all $i \in \N$, the fact that, for a given sequence $(c(n))_{n \in \N}$, $T^{c(n_1) + \cdots + c(n_i)} a \in E_i$ for all integers $i \leq n_1 < n_2 < \cdots < n_i$, implies that there exists a subsequence of $(c(n))_{n \in \N}$, that we call $(d(n))_{n \in \N}$ and points $x_i \in E_i$ for $i \in \N$ such that 
    \begin{equation*}
        \lim_{n \to \infty} T^{d(n)} x_i = x_{i+1} \text{ for all } i \in \N_0
    \end{equation*}
    where we consider $x_0=a$. The construction of $(d(n))_{n \in \N}$ is given by a diagonal argument and the fact that for any sequence $T^{c(n)} x_{i-1} \in E_i$ there is a subsequence which converge to some $x_{i} \in \overline{E_i} = E_i$. So the previous statement can be phrased in terms of \emph{infinite-length Erd\H{o}s progressions}.

     In Section \ref{sec application to Thue-Morse}, we show that \cref{U^infty very strong main theorem} cannot be improved to ensure that a $U^\infty(\Phi)$-uniform set is an IP-set (see \cref{prop infty uniform but not IP}).

\subsection{Relatively totally ergodic sets}

For an ergodic system, $\Xmt$ we define the \emph{rational Kronecker factor} $\cK_{rat}$, as the factor spanned by the rational eigenvalues of $\Xmt$ that is the functions $f \in L^2(\mu)$ such that $f \circ T =e(\alpha) f $ for some $\alpha \in \Q$. $\cK_{rat}$ is a factor of $Z_1(\mu)$. Whenever $\Xmt$ is \emph{totally ergodic}, that is $(X, \mu,T^i)$ is ergodic for all $i \in \N$, then $\cK_{rat}$ is trivial. 

 Following \cite{bergelson_moragues2022juxtaposing}, we say that a set $A \subset \N$ is a \emph{relatively totally ergodic set} if there exists an ergodic system $\Xmt$, a point $a \in \gen(\mu,\Phi)$ for some F\o lner sequence $\Phi$ and an open set $E \subset X$ with positive measure such that $\E(\1_E - \mu(E) \mid \cK_{rat})=0 $ and $A = \{ n \in \N \mid T^n a\in E\}$. The next theorem can be interpreted as a $k=1$ version of \cite[Conjecture 3.26]{kra_Moreira_Richter_Roberson2025problems}. 

\begin{proposition}
    Let $A \subset \N$ be a totally ergodic set then for every $k \in \N$ there exists an infinite set $B \subset \N$ such that $B^{\oplus k} \subset A$.
\end{proposition}

\begin{proof} 
   Let $\Xmt$, $\Phi = (\Phi_N)_{N\in \N}$, $a \in \gen(\mu, \Phi)$ and $E \subset X$ be given by the definition of relatively totally ergodic sets. By \cite[Lemma 5.8]{kmrr1}, we can suppose $\Xmt$ has continuous pronilfactor. Fixing $k \geq 2$, consider the measure $\sigma \in \cM(X^{k+1})$ defined in \cref{definition of sigma}. 
   
   Let $g  = \E(\1_E \mid Z_{k-1}) \in L^{\infty}(m_{k-1})$.  Notice that $\sigma(X^k \times E) = \int_{Z_{k-1}} g \diff \nu$ where $\nu$ is the projection of $\xi \in \cM(Z_{k-1}^k)$ to the last coordinate. By definition $\nu$ is the unique invariant measure of the orbit closure of $\pi_{k-1}(a) \in Z_{k-1}$ under the transformation $T^k$. In particular, since $\E_\mu (T^i\1_E \mid  \cK_{rat}) = \mu(E) $ and thus $ \E_{m_{k-1}}(T^ig \mid \cK_{rat}) = \mu(E)$ for all $i = 0, \ldots, k$, we get that $\int_{Z_{k-1}} g \diff \nu = \mu(E) >0$. Thus $\sigma(X^k \times E) >0$ and we conclude with \cref{Erdos cubes and progressions and sumsets} and \cref{lemma characterization progressive measures}.  
\end{proof}

\subsection{Intersecting uniform sets with $\nilbohrO{}$ sets} \label{sec intersection nilbor}

The goal of this section is to proof the following generalization of \cref{theorem dynamic reformulation} and give a combinatorial application generalizing \cref{theorem uniformity BB}.

\begin{theorem} \label{theorem dynamic reformulation with joining}
    Let $(X,\mu,T)$ be an ergodic system with topological pronilfactors, $a \in \gen(\mu, \Phi)$ for some F\o lner sequence $\Phi$, $k \geq 2$ and $E_1, \ldots, E_k \subset X$ be $U^k(X,\mu,T)$-uniform clopen sets. Let also $(Y, \nu, S)$ an $s$-step pronilfactor for some $s <k$, $y_0 \in Y$ and $V$ an open neighborhood of $y_0$.  There exists an infinite set $B \subset \N$ such that
         \begin{equation*}
              B^{\oplus i} \subset N((a,y_0),E_i \times V) \quad \text{for all } i =1, \ldots, k .
         \end{equation*} 
\end{theorem}

To prove \cref{theorem dynamic reformulation with joining} we study joinings of two ergodic systems and consequences in the $s$-step pronilfactors.

\subsection*{Joinings and pronilfactors}

For two measure preserving systems $(X, \mu,T)$ and $(Y, \nu, S)$ a \emph{joining} is a measure $\lambda \in \cM(X \times Y, T\times S)$ such that the first (resp. second) marginal equals $\mu$ (resp. $\nu$). We also say that the system $(X \times Y, \lambda, T \times S)$ itself is a joining. We denote $J_e(\mu, \nu) \subset \cM(X \times Y, T\times S)$ the set of ergodic joinings of $\mu$ and $\nu$. If $\mu$ and $\nu$ are ergodic then $J_e(\mu, \nu) $ is not empty.  

For this section we fix the notation, $\pi_{k}^X \colon X \to Z_k(\mu)$ and $\pi_{k}^Y \colon Y \to Z_k(\nu)$ to denote the corresponding factor maps to the $k$-step pronilfactor of two measure preserving systems $(X, \mu,T)$ and $(Y, \nu, S)$. Here we use again the regionally proximal relation of order $k$ introduced in \Cref{sec RPk}. 

\begin{proposition} \label{prop joining in the nil context}
    Let $s\in \N \cup \{ \infty\}$, $(X, \mu, T)$ and $(Y, \nu,S)$ be ergodic $s$-step pronilsystems. Suppose $\lambda \in J_e(\mu, \nu)$ is an ergodic joining. Then for every $k \leq s$, the $k$-step pronilfactor of $(X \times Y, \lambda, T \times S)$, $Z_k(\lambda)$ is isomorphic to a joining of $Z_k(\mu)$ and $Z_k(\nu)$. Moreover, the joining in $Z_k(\mu) \times Z_k(\nu)$ is given by the measure $(\pi_{k}^X \times \pi_{k}^Y)\lambda$. 
\end{proposition}

\begin{proof}
    Fix $k < s$.  By minimality, $\pi^X_k$ (resp. $\pi^Y_k$) is given by the quotient map by the equivalence relation $\RP^{[k]}(X)$ (resp. $\RP^{[k]}(Y)$), see \Cref{sec RPk}.
     Likewise, if we denote $W = \supp \lambda$, then by ergodicity of $\lambda$,  $(W, T \times S)$ is a minimal $s$-step pronilsystem nilsystem and therefore $Z_k(\lambda)$ coincides with $W/ \RP^{[k]}(W)$. Now using the definition of $\RP^{[k]}(W)$ with the induced distance in $W$ from $X \times Y$ we get
    \begin{equation*}
        \RP^{[k]}(W) = \{ (x,y,x',y') \in W \times W \mid (x,x') \in \RP^{[k]}(X) \text{ and } (y,y') \in \RP^{[k]}(Y) \}
    \end{equation*}
    which by the previous observation is given by $ \{ (x,y,x',y') \in W \times W \mid \pi^X_k(x)=\pi^X_k(x') \text{ and } \pi^Y_k(y)= \pi^Y_k(y') \}$. Therefore, the image of $W$ under the map $\pi^X_k \times \pi^Y_k$ is isomorphic to $W/\RP^{[k]}(W)= Z_k(\lambda)$.
    
    Thus, under the identification, $Z_k(\lambda) = (\pi^X_k \times \pi^Y_k)(W)$, $Z_k(\lambda)$ is a minimal subsystem of $Z_k(\mu) \times Z_k(\nu)$. Notice that $(\pi^X_k \times \pi^Y_k) \lambda$ is invariant under $T \times S$ and supported in  $Z_k(\lambda) $, so it is the unique invariant measure and its first  (resp. second) marginal coincide with  $m_{Z_k(\mu)}$ (resp. $m_{Z_k(\nu)}$), concluding the proof.  
\end{proof}

\begin{proposition} \label{prop joining of X and a pronil}
    Let $(X, \mu, T)$ be an ergodic system and let $(Y, \nu,S)$ be an ergodic $s$-step pronilsystems for some $s\in \N \cup \{ \infty\}$. Suppose $\lambda \in J_e(\mu, \nu)$ is an ergodic joining. Then for every $k \in \N \cup \{ \infty \}$, the $k$-step pronilfactor of $(X \times Y, \lambda, T \times S)$, $Z_k(\lambda)$ is isomorphic to a joining of $Z_k(\mu)$ and $Z_k(\nu)$. Moreover the joining in $Z_k(\mu) \times Z_k(\nu)$ is given by the measure $(\pi_{k}^X \times \pi_{k}^Y)\lambda$. 
\end{proposition}

\begin{proof}
    We divide the proof in cases. 
     First, let $k \geq s$, $f \in L^\infty (\mu)$ and $g \in L^\infty(\nu)$. Then 
    \begin{equation} \label{eq joining lambda}
        \E(f \otimes g \mid Z_{k}(\lambda)) = \E( f \otimes \1 \mid Z_k(\lambda)) \cdot (\1 \otimes g) =  (\E( f \mid Z_k(\mu)) \otimes \1 ) \cdot (\1 \otimes g)
    \end{equation}
    where in the first equality we use that $Y$ is an $s$-step pronilsystem, so it is a factor of $Z_k(\lambda)$ and thus $\1 \otimes g$ is measurable in $Z_k(\lambda)$. For the second equality we are using the definition of joining and the maximility of both $Z_k(\lambda)$ and $Z_k(\mu)$. 

    Noticing that $\pi_k^Y = \text{id}_Y$ we get that $\rho = (\pi_k^X \times \text{id}_Y) \lambda $ is clearly a joining of $m_{Z_k(\mu)}$ and $\nu$ and that by \eqref{eq joining lambda}, $(Z_k(\mu) \times Y, \rho, T \times S)$ is isomorphic to the $k$-step pronilfactor $Z_k(\lambda)$. This proves the proposition for $k \geq s$. 

    For $k < s$, notice that $Z_k(\lambda)$ is also the $k$-step pronilfactor of $Z_s(\lambda)$ which by the previous part is a joining of the pronilsystems $Z_s(\mu)$ and $Y$, so we can conclude using \cref{prop joining in the nil context}. 
\end{proof}

We directly deduce that. 

\begin{corollary} \label{cor cont factor map in the joining case}
    Let $(X, \mu,T)$ be an ergodic system with topological pronilfactors, $(Y, \nu,S)$ an ergodic $s$-step pronilfactor for some $s\in \N \cup \{\infty\}$ and $\lambda \in J_e(\nu, \mu)$ an ergodic joining. If for $k \in \N$, $\pi_k^X \colon X \to Z_k(\mu)$ and $\pi_k^Y \colon Y \to Z_{k}(\nu)$ are the continuous factor maps, then $\pi_k^X \times \pi_k^Y $ is a continuous factor map from $(X \times Y, \lambda, T \times S)$ to $Z_k(\lambda)$. 
\end{corollary}

With this we can prove \cref{theorem dynamic reformulation with joining}. 

\begin{proof}[Proof of \cref{theorem dynamic reformulation with joining}]

    By assumption,  $a \in \gen(\mu, \Phi)$ for some F\o lner sequence $\Phi$ and $y_0 \in Y$ is generic for $\nu$ along any F\o lner sequence, hence taking a sub-F\o lner sequence $\Phi'$ of $\Phi$, $(a,y_0)$ is generic along $\Phi'$ for some joining $\rho$ of $\mu$ and $\nu$. The support of the ergodic components of $\rho$ is contained in the support of $\rho$ ($\rho$-almost surely), hence $(a,y_0)$ is generic for some ergodic joining $\lambda \in J_e(\mu,\nu)$ along a F\o lner sequence $\Psi$. 

    By \cref{prop joining of X and a pronil}, $Z_{k-1}(\lambda)$ is a joining of $Z_{k-1}(\mu)$ and $Y$. In particular the continuous factor map $\tilde \pi_{k-1} \colon X \times Y \to Z_{k-1}(\lambda)$ is given by $\pi_{k-1}^X \times \text{id}_Y$ (see \cref{cor cont factor map in the joining case}). 
    Thus,  if $\tilde a = \pi^X_{k-1}(a)$, then $\tilde \pi_{k-1} (a, y_0) = (\tilde a , y_0)$. 

    Consider $\sigma_{(a,y_0)}$ the measure introduced in \cref{definition of sigma} for the system $(X\times Y, \lambda, T \times S)$ and the point $(a,y_0)$. Notice that for $V \subset Y$ and $E$ an $U^k (X, \mu,T)$-uniform set
    \begin{equation*}
        \E( \1_E \1_V \mid Z_{k-1}(\lambda))= \mu(E) (\1 \otimes \1_V).
    \end{equation*}
    With this 
    \begin{align*}
        \sigma_{(a,y_0)} ( (X \times Y) \times (E_1 \times V) \times \cdots \times (E_k \times V)) = \bigg(\prod_{i= 1}^k \mu(E_i) \bigg) \xi_{y_0} ( V \times \cdots \times V )
    \end{align*}
    where $\xi_{y_0}$ is the unique invariant measure of $W=\overline{\{ (S \times \cdots \times S)^n (y_0, \ldots, y_0) \mid n \in \Z\}}$. Since $W$ is minimal $V \times \cdots \times V$ is a neighborhood of $(y_0, \ldots, y_0) \in W$ we deduce that the last expression is possitive concluding the proof. 
\end{proof}

\subsection*{Combinatorial applications} 

A set $D \subset \N$ is a $\nilbohrO{s}$ for $s \in \N \cup \{\infty\}$ if there exists an $s$-step pronilsystem $(Z,T)$, a point $z_0 \in Z$ and an open neighborhood $U \subset X$ of $z_0$ such that $D = \{ n \in \N \mid T^n z_0 \in U\}$.   We also say that $D \subset \N$ is a $\nilbohrO{}$ if it is a $\nilbohrO{s}$ for some $s \in \N \cup \{\infty\}$. Examples of $\nilbohrO{}$ sets are, $r\N$ for $r \in \N$ and $\{ n \in \N \mid \{n^s \alpha \} \in (-\varepsilon, \epsilon) \}$ for some $\epsilon>0$, $s\in \N$, $\alpha \in \R \backslash \Q$. 

These sets were introduced by Host and Kra in \cite{host_kra2011nil_Bohr} generalizing the notion of $\bohr$ sets studied in \cite{bergelson_Furstenberg_Weiss2006piecewisebohr}. Notice that a $\nilbohrO{}$ set $D \subset \N$ is a central set (see \cite[Definition 8.3]{Furstenbergbook:1981}) and hence are $IP$-sets (see \cite{Furstenberg_Weiss_top_dynamics_combin_number_theory:1978}).  For further properties on $\nilbohrO{s}$ and their connection to combinatorial number theory refer to \cite{alweiss2025multiple_recurrence,bergelson_Leibman2018ipr,Huang_Shao_Ye_nilbohr_automorphy:2016}.

From \cref{theorem dynamic reformulation with joining} we deduce the following theorem.

\begin{theorem} \label{theorem uniformity BCD with nilbohr}  \label{theorem uniformity BB with nilbohr}
    Let $A_1, \ldots, A_k \subset \N$ be $U^k(\Phi)$-uniform sets for the same F\o lner sequence $\Phi$ and $D \subset \N$ a $\nilbohrO{s}$ set with $s <k$. There exists an infinite set $B \subset \N$ such that
         \begin{equation*}
              B^{\oplus i} \subset A_i \cap D \ \text{ for all } \ i =1, \ldots, k.
         \end{equation*}
\end{theorem}

    Notice that unlike $U^k(\Phi)$-uniformity, being $\nilbohrO{}$ is not shift invariant and from \cref{theorem summary} we see that is crucial to remain in a neighborhood of the original point. 
    One might want to generalize \cref{theorem uniformity BCD with nilbohr} by choosing distinct $\nilbohrO{s}$ sets $D_1, \ldots, D_k \subset \N$ and using $A_i \cap D_i$ instead of $A_i \cap D$. However, this does not give anything new, because for any finite family of $\nilbohrO{s}$ sets $D_1, \ldots, D_k \subset \N$, the intersection $D = D_1 \cap \cdots \cap D_k$ is also a $\nilbohrO{s}$ set. So the supposedly more general statement follows simply by noting that $A_i \cap D \subset A_i \cap D_i$ for all $i = 1, \ldots, k$.

\begin{proof}[Proof of \cref{theorem uniformity BCD with nilbohr}]

It suffices to show that for $U^k(\Phi)$-uniform sets $A_1, \ldots, A_k \subset \N$ and the $\nilbohrO{s}$ set $D \subset \N$, there exist systems and points as in \cref{theorem dynamic reformulation with joining} such that $N((a,y_0),E_i \times V) = A_1 \cap D$ for all $i=1, \ldots, k$.

For the sets $A_1, \ldots, A_k$ it suffices to take the ergodic system $(X, \mu,T)$, the generic point $a \in \gen(\mu,\Phi)$ and $U^k(X,\mu,T)$-uniform clopen sets $E_1, \ldots, E_k$ given in \cref{furstenberg correspondence for sets} such that $A_i = N(a, E_i)$ for all $i=1, \ldots,k$. We can suppose that $\Xmt$ has topological pronilfactors by \cite[Lemma 5.8]{kmrr1}. 

We know that, by definition of $\nilbohrO{}$ set, there exist a pronilsystem $(Y, \nu,S)$, a point $y_0 \in Y$ and a neighborhood $V$ of $y_0$ such that $N(y_0, V) =D$. With that we have constructed the systems, sets and points that fulfill all the assumptions and we conclude by noticing that for all $i = 1, \ldots, k$
\begin{equation*}
    N((a,y_0),E_i \times V) = N(a,E_i ) \cap N(y_0,V ) = A_i \cap D. \qedhere
\end{equation*} 
\end{proof}

Another consequence can be derive from the characterization of multiple recurrence sets done by Huang, Shao and Ye in \cite{Huang_Shao_Ye_nilbohr_automorphy:2016}.

\begin{corollary} 
    Let $A_1, \ldots, A_k \subset \N$ be $U^k(\Phi)$-uniform sets for the same F\o lner sequence $\Phi$. Let $(X,T)$ be a minimal system, $U\cap X$ be a non-empty open set, $d <k$ and $R=\{ n \in \N \mid U \cap T^{-n}U \cap \cdots \cap T^{-dn}U \} $. 
    There exists an infinite set $B \subset \N$ such that 
         \begin{equation*}
             B^{\oplus i} \subset A_i \cap D \ \text{ for all } \ i =1, \ldots, k.
         \end{equation*}
\end{corollary}

\begin{proof}
    Fix $A_i$ as in the statement.  By \cite[Theorem A (ii)]{Huang_Shao_Ye_nilbohr_automorphy:2016}, there exists a $\nilbohrO{d}$ $D$ such that $\diff_\Phi(R \Delta D) =0$. Consider $N = R \setminus D$, then $A_i \cap R \supset (A_i \setminus N) \cap D$ so it is enough to prove that $A_i \setminus N$ is $U^k(\Phi)$-uniform. Indeed if $N' = A_i \cap N$ then $\diff_\Phi(N') =0$ and $A_i = N' \cup (A_i \backslash N)  $ so
    \begin{align*}
        \norm{\1_{A_i \setminus N} - \diff_\Phi(A_i)   }_{U^k(\Phi)} = \norm{\1_{A_i} -  \1_{N'} - \diff_\Phi(A_i)   }_{U^k(\Phi)}& \\
        \leq \norm{\1_{A_i} - \diff_\Phi(A_i)   }_{U^k(\Phi)} + \norm{\1_{N'}  }_{U^k(\Phi)}.
    \end{align*}
     Thus, if we prove that $\norm{\1_{N'}  }_{U^k(\Phi)} =0$ then we can conclude. 
     
     This is a general phenomenon,  if $M \subset \N$ is a set such that $\diff_{\Phi}(M)=0$ for some F\o lner sequence $\Phi$, then $\norm{\1_M}_{U^k (\Phi)}=0$ for all $k \in \N$. 
Indeed,  $\diff_\Phi (M) = \norm{\1_M}_{U^0 (\Phi)}$ and by induction, since for every $h \in \N$, $\Delta_h \1_M = \1_M \cdot \1_{M-h} = \1_{M \cap (M-h)} $ and  $M \cap (M-h) \subset M$, hence  $\diff_\Phi (M \cap (M-h))=0$. Thus, we get $\norm{\Delta_h \1_M}_{U^k(\Phi)} = \norm{\1_{M \cap (M-h)}}_{U^k(\Phi)}= 0$ for all $h \in \N$ and therefore $\norm{\1_{M}}_{U^{k+1}(\Phi)}= 0$. 
\end{proof}

\subsection{Linear patterns with complexity $k$} \label{sec leibman configuration and other linear}

In this section we derive further sumset patterns using the notion of \emph{ complexity} for polynomials equations introduced in \cite{leibman2010complexityofpatterns} (see also \cite{kuca2023polynomial_complexity} for discussion on different notions of complexity). The main theorem of this section is the following (the definitions are introduced in what follows).

\begin{theorem} \label{theorem uniformity with Leibman criterion}
    Let $k,\ell \in \N$, $\Phi$ be a F\o lner sequence and $\cV \subset \hkbraket{\ell}^*$ be a family of vertex with complexity $k-1$. For any $U^{k}(\Phi)$-uniform sets $A_\epsilon \subset \N$ for $\epsilon \in \cV$ there exists infinite sets $B_1, \ldots, B_\ell$ such that
    \begin{equation*}
        \sum_{i=1}^\ell \epsilon_i \cdot B_i   \subset A_\epsilon \ \text{ for all } \epsilon \in \cV.
    \end{equation*}
\end{theorem}
Before introducing the terms we can derive two major consequences. Fix $\ell \geq 2$
\begin{enumerate}
    \item The vertex family $\cV = \hkbraket{\ell}^*$ is of complexity $\ell-1$. Thus, for any $U^\ell (\Phi)$-uniform sets $A_\epsilon \subset \N $ with $\epsilon \in \hkbraket{\ell}^*$ there exist $B_1, \ldots, B_\ell$ such that 
    \begin{equation*}
        \sum_{i=1}^\ell \epsilon_i B_i \subset A_\epsilon  \quad \text{ for all } \epsilon \in \hkbraket{\ell}^*.
    \end{equation*}
    \item Consider $\cV = \{ \epsilon  \in \hkbraket{\ell} \colon  |\epsilon| =1,\ell \}$, then $\cV$ is of complexity $1$, and hence for any $U^2 (\Phi)$-uniform sets $A_1, \ldots, A_\ell, A_{\ell+1} \subset \N $ there exist infinite sets $B_1, \ldots, B_{\ell}\subset \N$ such that 
    \begin{align*}
        &B_i \subset A_i &\text{for } i\in [\ell ]\\
        &B_1 + \cdots + B_\ell \subset A_{\ell+1}&
    \end{align*} 
\end{enumerate}

For a family of distinct non-constant polynomial with integer coefficients $\cP = \{p_j \colon \Z^\ell \to \Z \mid j=1,\ldots,r\}$, we say that $\cP \cup \{0\}$ is of \emph{complexity} $k-1$ if for every ergodic system $(X, \mu,T)$ and for every $f_1, \ldots, f_r \in L^{\infty} (\mu)$
     \begin{equation} \label{eq Leibman complexity}
        \lim_{N \to \infty} \bigg\lVert \frac{1}{N^\ell} \sum_{n \in \{1, \ldots, N\}^{\ell}} \prod_{j=1}^r T^{p_j(n)} f_j  - \frac{1}{N^\ell} \sum_{n \in \{1, \ldots, N\}^{\ell}} \prod_{j=1}^r T^{p_j(n)}  \E(f_j \mid Z_{k-1}) \bigg\rVert_{L^2(\mu)} =0.
    \end{equation}

    In the linear case, a useful criterion was derived by Leibman in \cite{leibman2010complexityofpatterns} and coincides with the \emph{true complexity} notion due to Gowers and Wolf \cite{gowers_wolf2010true}. A family of linear forms $\cL = \{ L_j \colon \Z^\ell \to \Z \mid j=1, \ldots, r \}$, has complexity $k-1$ if and only if $k$ is the minimal integer such that the polynomial vectors $(1, L_1(n), \ldots,(L_1(n))^k), \ldots,$ $ (1, L_r(n), \ldots,(L_r(n))^k) $ are linearly independent.

Now if we stick to the case of linear forms $L_j \colon \Z^\ell \to \Z$ with coefficients $0$ or $1$, then for all $n \in Z^\ell$, $L_j(n) = \epsilon \cdot n$ for some $\epsilon = \epsilon(j) \in \hkbraket{\ell}^*$. We denote such linear form $L_j = V_{\epsilon}$, $(\epsilon = \epsilon(j))$.  
In particular for every $g \in L^2(\mu)$, by the convergence of cubic averages in \cite{Host_Kra_nonconventional_averages_nilmanifolds:2005} and by \eqref{eq Leibman complexity}, 
\begin{align*}
    \int_{X^{\hkbraket{\ell}}} g^{[\underline{0}]} \cdot \prod_{j=1}^r  f_j^{[\epsilon(j)]}  \diff \mu^{\hkbraket{\ell}} &= \lim_{N \to \infty} \frac{1}{N^\ell} \sum_{n \in \{1, \ldots, N\}^{\ell}} \int_X g \cdot  \prod_{j=1}^r T^{L_j(n)} f_j  \diff \mu \\ 
    &= \lim_{N \to \infty} \frac{1}{N^\ell} \sum_{n \in \{1, \ldots, N\}^{\ell}} \int_X g \cdot  \prod_{j=1}^r T^{L_j(n)} \E(f_j \mid Z_{k-1})  \diff \mu \\
    &= \int_{X^{\hkbraket{\ell}}} g^{[\underline{0}]} \cdot \prod_{j=1}^r  \E(f_j \mid Z_{k-1})^{[\epsilon(j)]}  \diff \mu^{\hkbraket{\ell}}.
\end{align*}

\begin{definition}
    For $\ell, k \in \N$, we say that a set of vertex $\mathcal{V} \subset \hkbraket{\ell}^*$ is of complexity $k-1$ if the family of linear functions $\cL = \{ V_\epsilon \mid \epsilon \in \mathcal{V} \} \cup \{0 \}$ is of complexity $k-1$. 
\end{definition}

By the previous discussion, if $\mathcal{V} \subset \hkbraket{\ell}^*$ is of complexity $k-1$ then for all $g \in L^2(\mu)$ and $f_\epsilon \in L^\infty (\mu)$, $\epsilon \in \mathcal{V}$  
\begin{equation} \label{eq def vertex complexity}
    \int_{X^{\hkbraket{\ell}}} g^{[\underline{0}]} \cdot \prod_{\epsilon \in \cV}  f_\epsilon^{[\epsilon]}  \diff \mu^{\hkbraket{\ell}} = \int_{X^{\hkbraket{\ell}}} g^{[\underline{0}]} \cdot \prod_{\epsilon \in \cV}  \E(f_\epsilon \mid Z_{k-1})^{[\epsilon]} \diff \mu^{\hkbraket{\ell}}
\end{equation}

The next proposition states that the equality in \eqref{eq def vertex complexity} is maintained under disintegration.

\begin{proposition} \label{theorem dynamic version of leibman complexity thrm}
Let $(X, \mu,T)$ be an ergodic system and $a \in \gen(\mu, \Phi)$ for some F\o lner sequence $\Phi$. 
 Take $f_\epsilon \in C(X)$ for $\epsilon \in \mathcal{V} \subset \hkbraket{\ell}^*$ a set of vertex of complexity $k-1$, then
     \begin{equation} \label{eq equality measure sigma[k]}
         \int_{X^{\hkbraket{\ell}}} \prod_{\epsilon \in \cV}  f_\epsilon^{[\epsilon]} \diff \sigma_a^{\hkbraket{\ell}} = \int_{X^{\hkbraket{\ell}}} \prod_{\epsilon \in \cV} \E(f_{\epsilon} \mid Z_{k-1})^{[\epsilon]} \diff \sigma_a^{\hkbraket{\ell}}
     \end{equation}
     where the measure $\sigma_a^{\hkbraket{\ell}}$ is given in \cref{definition of sigma cubic}.
\end{proposition}

\begin{proof}

   By definition, since $\underline{0} \not \in \cV$,
   \begin{equation*}
       \int \prod_{\epsilon \in \cV} f_\epsilon^{[\epsilon]} \diff \sigma^{\hkbraket{\ell}}_a =   \int \prod_{\epsilon \in \cV} \E(f_\epsilon^{[\epsilon]}\mid Z_{\ell-1}) \diff \xi^{\hkbraket{\ell}}_a
   \end{equation*}

    Consider first $(X, \mu,T)$ to be an $(\ell-1)$-step pronilsystem. By continuous disintegration (see \cref{prop cont disintegration and conditional expectation} and \cref{prop cont k-step max pronilfactors})  $g_\epsilon = \E(f_{\epsilon} \mid Z_{k-1})$ is continuous for all $\epsilon \in \cV$. By \eqref{eq def vertex complexity} for $g = \1$, we get that 
    \begin{equation*}
          \int_{X^{\hkbraket{\ell}}} \prod_{\epsilon \in \cV}  f_\epsilon^{[\epsilon]} \diff \mu^{\hkbraket{\ell}} = \int_{X^{\hkbraket{\ell}}} \prod_{\epsilon \in \cV} g^{[\epsilon]}_\epsilon  \diff \mu^{\hkbraket{\ell}}.
    \end{equation*}
    Since $g_\epsilon$ is continuous for all $\epsilon \in \cV$  and $y \mapsto \sigma^{\hkbraket{\ell}}_y$  is a continuous disintegration of $\mu^{\hkbraket{\ell}}$ (see \cite[Theorem 7.1]{kmrr1}), we conclude \eqref{eq equality measure sigma[k]}.

    Now for $X$ an arbitrary ergodic system, assume that $\norm{f_\epsilon}_\infty \leq 1$ for all $\epsilon\in \cV$. Consider $\eta>0$ and take $g_\epsilon \in C(Z_{\ell-1})$ such that $\norm{g_\epsilon - \E(f_\epsilon \mid Z_{\ell-1})}_{L^2(m_{\ell-1})} < \eta$ for every $\epsilon \in \cV$. Then since $\underline{0} \not \in \cV$, for $a \in \gen(\mu,\Phi)$

    \begin{align*}
        &\left| \int \prod_{\epsilon \in \cV} g_\epsilon \diff \xi^{\hkbraket{\ell}}_a  -  \int \prod_{\epsilon \in \cV} f_\epsilon  \diff \sigma^{\hkbraket{\ell}}_a \right| = \left| \int \prod_{\epsilon \in \cV} g_\epsilon \diff \xi^{\hkbraket{\ell}}_a  -  \int \prod_{\epsilon \in \cV} \E(f_\epsilon \mid Z_{\ell-1}) \diff \xi^{\hkbraket{\ell}}_a \right| \\
        \leq& \prod_{\epsilon \in \cV} \norm{g_\epsilon - \E(f_\epsilon \mid Z_{\ell-1})}_{L^2(\mu)} \leq |\cV| \cdot \eta
    \end{align*}
    where for the first inequality we use the triangular inequality $|\cV|$ times, then Cauchy-Schwarz inequality, the fact that the projection of $\xi_a$ to any coordinate $\varepsilon \in \hkbraket{\ell}^*$ is $m_{\ell-1}$ and that since $\norm{f_\epsilon}_\infty \leq 1$, $\E(f_\epsilon \mid Z_{\ell-1}) \leq 1$ almost surely. With this we can derive the general case by the first part. 
\end{proof}

From \cref{theorem dynamic version of leibman complexity thrm} we derive \cref{theorem uniformity with Leibman criterion}.

\begin{proof}[Proof of \cref{theorem uniformity with Leibman criterion}]

By \cref{theorem dynamic version of leibman complexity thrm} if $\{E_\epsilon \colon \epsilon \in \cV\}$ is a familly of $U^k(X, \mu,T)$-uniform sets and $\cV \subset \hkbraket{\ell}$ is a set of vertex with complexity $k-1$, then 
\begin{align*}
    \sigma^{\hkbraket{\ell}}_a\Big( \bigcap_{\epsilon \in \cV} \{\boldsymbol{y} \in X^{\hkbraket{\ell}} \mid y_\epsilon \in E_\epsilon\}\Big) &= \int_{X^{\hkbraket{\ell}}} \prod_{\epsilon \in \cV} \E(\1_{E_{\epsilon}} \mid Z_{k-1})^{[\epsilon]}  \diff \sigma_a^{\hkbraket{\ell}} = \prod_{\epsilon \in \cV} \mu(E_\epsilon) >0.
\end{align*} 
With that, using Lemmas \ref{lemma sigma^[k] property} (ii) and \ref{Erdos cubes and progressions and sumsets}, there exist infinite sets $B_1, \ldots, B_\ell \subset \N$ such that for all $\epsilon \in \cV$,  $\sum_{i=1}^\ell \epsilon_i B_i \subset N(a, E_\epsilon)$. We conclude as in \Cref{sec dynamic reformulation}. 
\end{proof}

\begin{remark*}
    With the previous proof and similar to the analysis from \Cref{sec intersection nilbor}, one can derive a version of \cref{theorem dynamic version of leibman complexity thrm} intersecting the $U^k(\Phi)$-uniform sets $A_\epsilon $ with a $\nilbohrO{s}$ set $D$, for $s < k$.    
\end{remark*}

\section{Applications from substitution subshifts} \label{sec constant length substitutions}

The goal of this section is firstly to prove \cref{constant length subs sumset intro} and exploit its consequences. Secondly, to show that \cref{theorem summary} cannot be generalized to any minimal system. For both result we need to introduce substitution subshifts (for a more systematic analysis on substitutions refer to \cite{allouche_shallit2003automatic} and \cite{Queffelec1987}) 

We say that a finite set $\cA$ is an alphabet and we denote $\cA^*$ the set of all finite words with letters in $\cA$. The set $\cA^*$ is a monoid for concatenation of words. A substitution $\sigma \colon \cA^* \to \cA^*$ is a concatenation endomorphism, that is $\sigma(uw) = \sigma(w) \sigma(u)$ for all $u,w \in A^*$. We say that a substitution $\sigma \colon \cA^* \to \cA^*$ is primitive if for all $a \in \cA$ there exists $k \in \N$ such that every letter $b$ appears in the word $\sigma^k(a)$. 

For a word $w=w_0 \cdots w_r \in \cA^*$ we denote its length $|w|=r+1$ (we highlight that every word is indexed from $0$). We say that a substitution $\sigma\colon \cA^* \to \cA^*$ has constant-length $\ell$ if $|\sigma(a)| = \ell $ for all $a\in \cA$. 

Take a primitive substitution $\sigma \colon \cA^* \to \cA^*$ and suppose that for some $a \in \cA$ the first letter of $\sigma(a)$ is $a$ and $|\sigma(a)| \geq 2$, then $\lim_{k \to \infty} |\sigma^k(a)|=\infty$ and we denote $u = \sigma^\infty(a) \in \cA^{\N_0}$ the infinite word such that $u_n =( \sigma^k(a))_n$ for all $n \in \N_0$ and big enough $k$. Likewise, if for some $b \in \cA$, the word $\sigma(b)$ ends with $b$, we consider the analogous two-sided infinite word $\sigma^\infty(b \bullet a) \in \cA^\Z$ (where the $\bullet$ symbol indicates the $0$ index on the right-hand-side of it). Under those assumptions, if $ba$ appears in $\sigma^k(c)$ for some $k \in \N$ and $c \in \cA$, then we denote $X(\sigma)  = \overline{\{ S^n \sigma^\infty(b \bullet a)  \mid n \in \Z\} } $ where $S\colon \cA^\Z \to \cA^\Z$ is the shift map. $(X(\sigma), S)$ is the \emph{substitution subshift given by} $\sigma$ and it does not depend on the choice of $ba$.  The system $(X(\sigma), S)$ is minimal and uniquely ergodic. 

\subsection{Uniform sequences arising from substitution subshifts} 
The following theorem implies \cref{constant length subs sumset intro}.

\begin{theorem} \label{thrm uniformity constant length substitution}
    Let $(X(\sigma), S)$ be a substitution subshift and $\sigma$ be a primitive constant-length substitution. Let also $f \in C(X(\sigma))$, $x \in X(\sigma)$ and define the sequence $\phi(n) = f(S^n x)$ for all $n \in \N_0$. If for every periodic sequence $(\psi(n))_{n \in \N}$
    \begin{equation} \label{eq orthogonal to periodic}
        \lim_{N \to \infty} \frac{1}{N} \sum_{n=1}^N \phi(n) \psi(n) =0
    \end{equation}
    then $\norm{\phi}_{U^k(\Phi)} =0$ for all $k \in \N$ and for any F\o lner sequence $\Phi$. 
\end{theorem}

For the proof, we introduce the height $h(\sigma)$ of a substitution. Given a primitive constant length substitution $\sigma \colon \cA^* \to \cA^*$,  $h(\sigma) = \max\{ n \in \N \mid (n,\ell)=1, n $ divides $\gcd\{i \colon u_i = u_0\}\}$, where $u = \sigma^\infty(a)$ for some $a \in \cA$. Notice that $h(\sigma) \leq |\cA|$. For this work the value of $h(\sigma)$ is not particularly important, we only use it to characterize the Kronecker factor. In later examples $h(\sigma)=1$.
Before proving \cref{thrm uniformity constant length substitution}, we conveniently introduce the following lemma, that can be found in the proof of \cite[Proposition 2.21]{Host_Kra09}. 

\begin{lemma} \label{lemma key for constant length nilfactors}
    $(X(\sigma), \mu, S)$ be a substitution subshift with $\sigma \colon \cA ^* \to \cA^*$ a primitive constant length substitution. If $Z_k$ denote the $k$-step pronilfactor of $(X(\sigma), \mu, S)$ for $k \in \N$, then $Z_1 = Z_k$ for all $k \in \N$ and the factor map $\pi \colon X(\sigma) \to Z_1$ is continuous. 
\end{lemma}

\begin{proof}
    Without lost of generality we assume that $X(\sigma)$ is infinite. The fact that $\pi \colon X(\sigma) \to Z_1$ is continuous is a direct consequence of \cite[Theorem 6.2]{Queffelec1987}. Now, we denote by $\ell$ the length of the substitution $\sigma \colon \cA^* \to \cA^*$. 
    
    By Dekking's result \cite[Theorem 13]{Dekking1978}, the Kronecker factor $Z_1$ is isomorphic to $(\Z/h\Z) \times \Z_{(\ell)}$ where $\Z_{(\ell)} = \{ z\in \prod_{n\in \N}(\Z/\ell^n\Z) \mid z_n = z_{n+1} \mod \ell^n\} $ denotes the $\ell$-odometer and $h=h(\sigma)$. 
    
    If one considers $p \colon Z \to \Z_{(\ell)} $ the projection to the second coordinate in $(\Z/h\Z) \times \Z_{(\ell)}$ and $\theta = p \circ \pi$, we get that $\theta \colon X(\sigma) \to \Z_{(\ell)}$ is also a continuous factor map and it is almost $|\cA|$-to-one. In particular, $\pi$ is almost finite-to-one, and since the $2$-step pronilfactor of $(X(\sigma), \mu, S)$ is an extension of $Z_1$ by a connected abelian group (see \cite{Host_Kra_nonconventional_averages_nilmanifolds:2005}), it has to be equal to $Z_1$. Thus, $Z_k = Z_1$ for all $k \in \N$.
\end{proof}

\begin{proof}[Proof of \cref{thrm uniformity constant length substitution}]
    Let $(X(\sigma), \mu, S)$ be a substitution subshift as in the statementand $Z = (\Z/h\Z) \times \Z_{(\ell)}$ the Kronecker factor that coincides with the $\infty$-step pronilfactor by \cref{lemma key for constant length nilfactors}. Let also $\pi \colon X(\sigma) \to Z$ be the continuous factor map.  

    The topology of $Z$ is generated by the clopen sets $\{i\} \times E_{m,j}$ where $i \in \Z/h\Z$ and $E_{m,j}=\{ z \in \Z_{(\ell)}  \mid z_m=j\} $ for $m \in \N$ and $0\leq j <\ell^m$. Notice that under the action $R \colon Z \to Z $ given by $(y,z) \mapsto (y,z) + (1,1) $, fixing arbitrary $i \in \Z/h\Z$, $m \in \N$, $0\leq j <\ell^m$,  and $(y,z) \in Z$ the set $\{ n \in \N \mid R^n(y,z) \in \{i\} \times E_{m,j} \} = (h\N+i') \cap (2^m \N + j')$ for some $0\leq i'<h$ and $0\leq j' <2^m$ (depending on $(y,z)$). 
    
    With this, consider $f \in C(X(\sigma))$ and $x \in X(\sigma)$ as in the statement. For all, $i \in \Z/h\Z$, $m \in \N$ and $0\leq j <\ell^m$, $\1_{\pi^{-1}(\{i\} \times E_{m,j})}$ is a continuous function, hence since by unique ergodicity $x$ is generic we get that

    \begin{align*}
        \int_{X(\sigma)} f (x) \1_{\{i\} \times E_{m,j}} (\pi(x)) \diff \mu(x) 
         =& \lim_{N \to \infty} \frac{1}{N} \sum_{n=0}^{N-1} f(S^n x) \1_{\{i\} \times E_{m,j}}(R^n \pi (x)) \\ 
        =&\lim_{N \to \infty} \frac{1}{N} \sum_{n =0}^{N-1} \phi(n) \psi(n)
    \end{align*}  
    where $\psi $ is the periodic function $ \1_{(h\N+i') \cap (\ell^m \N + j')}$ for some  $0\leq i'<h$ and $0\leq j' <\ell^m$ depending on $\pi(x)$. Therefore, by \eqref{eq orthogonal to periodic} we get that
    \begin{align*}
        \int_X f (x) \1_{\{i\} \times E_{m,j}} (\pi(x)) \diff \mu(x) = 0
    \end{align*} 
    and since the family of clopen sets $\{i\} \times  E_{m,j}$ generates the Borel sigma algebra in $Z$ we get that $\E(f \mid Z) =0$. With the previous analysis this implies that $\norm{f}_{U^k(X(\sigma), \mu, T)} =0$ for all $k \in \N$. Thus by construction and unique ergodicity  $\norm{\phi}_{U^k(\Phi)} =0$ for any F\o lner sequence $\Phi$.
\end{proof}

\begin{remark} \label{remark periodicity needed for constant legnth}
    With the previous proof, given a constant length substitution $\sigma$ if one computes the value of $h=h(\sigma)$, then one gets the same conclusion of \cref{thrm uniformity constant length substitution} by only checking \eqref{eq orthogonal to periodic} with sequences of the form $\psi(n) = \1_{h\N +i}(n) \1_{\ell^m\N + j }(n)$ for $i =0, \ldots, h-1$, $m \in \N$ and $j =0, \ldots, \ell^m-1$. As we see in the next proof this is also the case for \cref{constant length subs sumset intro}. 
\end{remark}

\begin{proof}[Proof of \cref{constant length subs sumset intro}]
    If $E \subset X(\sigma)$ is a clopen set, then the function $f = \1_E - \mu(E)$ is continuous. Taking $\phi(n) = f(S^nx) $ for $x \in X(\sigma)$ as in the assumption and by \cref{thrm uniformity constant length substitution}, we get that $\norm{\phi}_{U^k(\Phi)}=0$ for all $k \in \N$. This by definition implies that  $A= \{ m \in \N \mid S^mx\in E\} $ is $U^\infty(\Phi)$-uniform. The sumset part of \cref{constant length subs sumset intro} is a consequence of \cref{U^infty dynamical}. 
\end{proof}

\begin{remark*}
    Notice that under the assumptions of \cref{thrm uniformity constant length substitution}, by unique ergodicity of the system one deduces that
    \begin{equation*}
        \norm{\phi}_{U^k(\Phi)} = \lim_{H \to \infty} \frac{1}{H^k} \sum_{h \in \{1, \ldots, H\}^k} \bigg| \lim_{N \to \infty} \frac{1}{|\Phi_N|} \sum_{n \in \Phi_N} \prod_{\epsilon \in \hkbraket{k}} \phi(n + \epsilon \cdot h) \bigg|^{2^k} =0. 
    \end{equation*}
    In particular, from the next section, for the Thue-Morse sequence $(t(n))_{n \in \N_0}$ (see \cite[Proposition 3.3]{baake_coons2024correlations_Thue_Morse} for the case $k=1$) 
    \begin{equation*}
    \lim_{H \to \infty} \frac{1}{H^k} \sum_{h \in \{1, \ldots, H\}^k} \bigg| \lim_{N \to \infty} \frac{1}{N} \sum_{n =0}^{N-1} \prod_{\epsilon \in \hkbraket{k}} (t(n + \epsilon \cdot h) -1/2) \bigg|^{2^k} =0. 
\end{equation*}
\end{remark*}

\subsection{$U^\infty(\Phi)$-uniformity for subwords of the Thue-Morse sequence} \label{sec application to Thue-Morse}

Consider $(t(n))_{n \in \N_0}$ the Thue-Morse sequence defined in the introduction. Consider also $\tau \colon \{0,1\}^* \to \{0,1\}^*$ the substitution given by $\tau(0)=01$ and $\tau(1)=10$, then  $t(n) = (\tau^\infty(0))_n$ for all $n \in \N_0$. The morphism $\tau$ is called the Thue-Morse substitution and $(X(\tau),S)$ is the Thue-Morse subshift.

For a finite alphabet $\cA$, we say that a word $w\in \cA^{r+1}$ is a prefix of a word $u \in A^{s+1}$ for $s \geq r$, if $w_i=u_i$ for all $i =0, 1,\ldots,r$.  For a word $w \in \{0,1\}^*$ denote $A_w=\{ m \in \N \mid t(m)=w_0, \ldots, t(m+r) = w_r\}$. The main result of this section is the following.

\begin{proposition} \label{thrm infty uniformity thue-morse}
     If $w$ is a prefix of $\tau^k(0)$ or $\tau^k(1)$ for some $k \in \N$, then $A_w / 2^k \N = \{ n \in \N \mid 2^kn \in A_w\}$ is $U^{\infty}(\Phi)$-uniform for any F\o lner sequence $\Phi$.
\end{proposition}

 With this we recover part of \cite[Theorem 3.1]{bucci_Hindman_Puzynina_Zamboni2013additive_thue_morse}.
\begin{corollary}  \label{cor sumset thue-morse}
    If $w$ is a prefix of $\tau^k(0)$ or $\tau^k(1)$ for some $k \in \N$, then there exists $B = \{ b_1 < b_2 < \ldots\} \subset \N$ infinite such that 
    \begin{equation*}
        \bigcup_{k \in \N} (B \setminus \{ b_1, \ldots, b_{k-1}\})^{\oplus k} \subset A_w
    \end{equation*}
\end{corollary}

\begin{proof}
    Fix $w$ and $k \in \N$ as in the statement, by \cref{thrm infty uniformity thue-morse}, $A_w/2^k \N $ is $U^{\infty}(\Phi)$-uniform, hence using \cref{U^infty dynamical} we get that there exists an infinite set $C = \{ c_1 < c_2 < \cdots \} \subset \N$ such that 
    \begin{equation*}
        \bigcup_{k \in \N} (C \setminus \{ c_1, \ldots, c_{k-1}\})^{\oplus k} \subset A_w/2^k\N.
    \end{equation*}
    Thus we conclude by taking $B= 2^k C$. 
\end{proof}

    \begin{corollary} \label{prop infty uniform but not IP}
        There exists a $U^\infty(\Phi)$-uniform set that is not an IP-set.
    \end{corollary}

    \begin{proof}
        By \cite[Corollary 3.6]{bucci_Hindman_Puzynina_Zamboni2013additive_thue_morse} the set $A_1$ is not an IP-set and by \cref{thrm infty uniformity thue-morse} it is $U^\infty(\Phi)$-uniform for all F\o lner sequence $\Phi$.
    \end{proof}

To prove \cref{thrm infty uniformity thue-morse}, we need the following lemma. 

\begin{lemma} \label{lemma frequency letter thue-morse diadics}
     For all $a =0,1$, $k \in \N$ and $0 \leq j < 2^k $,
    \begin{equation*} 
        d(A_a \cap (2^k \N + j )) = \frac{1}{2^{k+1}}.
    \end{equation*}
\end{lemma}

\begin{proof}
    Fix $k \in \N$, $0 \leq j < 2^k $ and $a \in \{0,1\}$, then
    \begin{align} \label{eq thue 2k}
        d(A_a \cap (2^k \N+j)) = \lim_{N \to \infty} \frac{1}{2^k \cdot N} \sum_{n = 0}^{N-1} \1_{t(2^kn +j) = a} 
    \end{align}
     By substitution rule, if we denote $u(n)$ the word $\tau^{k}(t(n))$ we get that $t(2^kn +j ) = u(n)_j$ and so $t(2^kn +j) = a$ if and only if $a = u(n)_j$. Notice that the value of $u(n)_j$ is uniquely determined by the value of $t(n)$ and there is only one letter $b \in \{0,1\}$ (that depends on $a,k$ and $j$ but does not depend on $n$), for which $(\tau^{k}(b))_j = a$. Thus $t(2^kn +j) = a$ if and only if $t(n) = b$. With this, replacing on \eqref{eq thue 2k},
    \begin{align*}
        d(A_a \cap (2^k \N+j)) = \lim_{N \to \infty} \frac{1}{2^k \cdot N} \sum_{n = 0}^{N-1} \1_{t(n) = b}  = \frac{1}{2^{k+1}}
    \end{align*}
    where the last equality is given by the frequency of the letters in the Thue-Morse word. 
\end{proof}

With this we prove the special case of \cref{thrm infty uniformity thue-morse} when $w=0$ or $w=1$.

\begin{corollary} \label{thrm intro infty uniformity thue-morse}
    For $a=0,1$, $A_a$ is $U^{\infty}(\Phi)$-uniform for any F\o lner sequence $\Phi$.
\end{corollary}

\begin{proof} For the Thue-Morse subshift $h(\tau) =1$, then by \cref{constant length subs sumset intro} and \cref{remark periodicity needed for constant legnth} taking $\mathbf{t} = (\tau^2)^\infty(1 \bullet0)$ we only need to check that 
\begin{equation*}
    \lim_{N \to \infty} \frac{1}{N} \sum_{n=1}^N (\1_{C_a}(S^n \mathbf{t}) - \mu(C_a)) \1_{2^m\N + j} (n)=0
\end{equation*}
where $C_a = \{ y \in X(\tau) \mid y_0 = a\}$ for $a = 0,1$ and $m \in \N$, $0 \leq j < 2^m$. Since $\mu(C_a) = 1/2$, this is equivalent to check that
\begin{equation} \label{eq dentro de la dem}
   \lim_{N \to \infty} \frac{1}{N} \sum_{n=1}^N \1_{C_a}(S^n \mathbf{t}) \1_{2^m\N + j} (n)=1/2^{m+1}
\end{equation}
since $A_a = \{ n \in \N \mid S^n \mathbf{t} \in C_a\}$, \eqref{eq dentro de la dem} is ensured by \cref{lemma frequency letter thue-morse diadics}.
\end{proof}

\begin{lemma} \label{lemma equality prefix suffix}
    Let $a \in \{0,1\}$. If $w$ is a prefix of $\tau^k(a)$ for some $k \in \N$, then $A_w \cap 2^k\N = A_a \cap 2^k\N$.
\end{lemma}

\begin{proof}
    As was previously exploited, $t(2^kn) = t(n)$ and more generally for all $0 \leq j < 2^k$, $t(2^kn+j) = \tau^k(t(n))_j$. In particular the values of $t(2^kn+j)$ are determined by the value of $t(2^kn)$ so that $t(2^kn) t(2^kn+1) \cdots t(2^k + r)$ is the prefix of length $r+1 < 2^k$ of the word $\tau^k(t(2^kn))$ concluding the proof.  
\end{proof}

\begin{proof}[Proof of \cref{thrm infty uniformity thue-morse}]
    For $a \in \{0,1\}$ the set $A_a$ is $U^{\infty}(\Phi)$-uniform for any F\o lner sequence $\Phi$, then $A_a/2^k\N$ is $U^{\infty}(\Psi)$-uniform for any F\o lner sequence $\Psi$. Since $A_a/2^k\N = (A_a \cap 2^k\N) /2^k\N$ we conclude by \cref{lemma equality prefix suffix}. 
\end{proof}

\begin{remark*}
    In this section we only exploited some general properties of the Thue-Morse substitution. In particular, instead of  $(t(n))_{n \in \N_0}$ one can consider a sequence $(s(n))_{n \in \N_0}$ given by $s(n) = (\sigma^\infty (a))_n$ for all $n \in \N_0$ where $a$ is a letter in a finite alphabet $\cA$ and $\sigma\colon \cA^* \to \cA^*$ is a primitive constant-length one-to-one on letters substitution with $h(\sigma)=1$. These technical terms are defined in \cite[Chapter 6]{Queffelec1987}.
\end{remark*}

We finish this section by proving a version of \cref{thrm intro infty uniformity thue-morse} and \cref{cor sumset thue-morse} for the Rudin-Shapiro sequence $(r(n))_{n \in \N_0}$ defined in the introduction. Denote $R_a = \{ n \in \N \mid r(n)=a\}$ for $a =0,1$. 

\begin{proposition} \label{prop sumset and uniformity Rudin-shapiro}
    The set $R_a$ is $U^\infty(\Phi)$-uniform for any F\o lner sequence $\Phi$. In particular contains the pattern described in \eqref{eq pattern un constant length subs}.
\end{proposition}

\begin{proof}
    Let $\zeta \colon \{0,1,2,3\}^* \to \{0,1,2,3\}^* $ be the substitution given by $\zeta(0)=01$, $\zeta(1)=02$, $\zeta(2)=31$ and $\zeta(3)=32$. In $(X(\zeta),\mu,S)$ consider $\mathbf{r}= (\zeta^2)^\infty(1\bullet0)$ and the continuous function $f \in C(X(\sigma))$, $f(x) = 0$ if $x_0 \in \{0, 1\}$ and $f(x) = 1$ if $x_0 \in \{2,3\}$. We get that $f(S^n \mathbf{r}) = r(n)$ for all $n \in \N$ and therefore $R_a = \{ n \in \N \mid f(S^n \mathbf{r}) = a\}$ for $a =0,1$. We prove the case $a=1$ noticing that the case $a=0$ is identical.

    Since $h(\zeta) = 1$, using \cref{constant length subs sumset intro} and \cref{remark periodicity needed for constant legnth}, we need to check that
    \begin{equation*}
        \lim_{N \to \infty} \frac{1}{N} \sum_{n=1}^N f(S^n\mathbf{r}) \1_{2^k\N +j}(n) = 1/2^{k+1}
    \end{equation*}
    for all $k \in \N$ and $0 \leq j < 2^k$ to conclude the proposition. 
    
    Fix $k \in \N$ and $0 \leq j < 2^k$. Notice that for $n$ even $(S^n \mathbf{r})_0 \in \{0,2\}$ and for $n$ odd $(S^n \mathbf{r})_0 \in \{1,3\}$, with this one can split the proof based on the parity of $j$. Without lost of generality take $j$ even:
    \begin{align*}
        \lim_{N \to \infty} \frac{1}{N} \sum_{n=1}^N f(S^n\mathbf{r}) \1_{2^m\N +j}(n) = \lim_{N \to \infty} \frac{1}{2^k \cdot N} \sum_{n = 0}^{N-1} f(S^{2^kn+j}\mathbf{r})  \\ 
        = \lim_{N \to \infty} \frac{1}{2^k \cdot N} \sum_{n = 0}^{N-1} \1_{(S^{2^kn+j}\mathbf{r})_0 = 2 }.
    \end{align*}
    By substitution rule the letter $(S^{2^kn+j}\mathbf{r})_0$ depends on $(S^{n}\mathbf{r})_0$. Now, by the shape of the substitution $\zeta$, there are exactly two letters let say $a,b \in \{0,1,2,3\}$, for which $(\zeta^k(a))_{j} = (\zeta^k(b))_{j}= 2$. In other words $(S^{2^kn+j}\mathbf{r})_0 = 2$ if and only if $(S^{n}\mathbf{r})_0 = a$ or $b$. Thus,
    \begin{align*}
        \lim_{N \to \infty} \frac{1}{2^k \cdot N} \sum_{n = 0}^{N-1} \1_{(S^{2^kn+j}\mathbf{r})_0 = 2 }  =  \lim_{N \to \infty} \frac{1}{2^k \cdot N} \sum_{n = 0}^{N-1}( \1_{(S^{n}\mathbf{r})_0 = a } + \1_{(S^{n}\mathbf{r})_0 = b } ) = 1/2^{k+1}
    \end{align*}
    concluding the proof.
\end{proof}

\subsection{A counterexample to generalizations of \cref{theorem summary}}

 The counterexample in this section is based in \cite{bucci_Hindman_Puzynina_Zamboni2013additive_thue_morse}.

 \begin{proposition} \label{prop counterex gen BB}
     There exist a minimal and uniquely ergodic system $(X,\mu,T)$ with topological pronilfactors, two points $x,y \in X$ and $V$ a neighborhood of $y$ such that $(x, y) \in \RP^{[1]}(X)$ but there is no infinite $B \subset \N$ with $B \oplus B \subset N(x,V)$.
 \end{proposition}

\begin{proof}

Consider the substitution $\sigma \colon \{0,1,2\}^* \to \{0,1,2\}^*$ given by $\sigma(0)=02$, $\sigma(1)=20$ and $\sigma(2) = 10$ and the minimal substitution subshift $(X(\sigma), S)$. As a consequence of \cite[Theorem 13]{Dekking1978} the maximal Kronecker factor of $(X(\sigma), \mu, S)$ is the $2$-odometer $(\Z_{(2)}, \nu,R)$, where $\mu$ is its unique invariant measure. Notice that the map $\tilde \theta \colon \{0,1,2\} \to \{0,1\}$ that sends $0 \mapsto 0$ and $1,2 \mapsto 1$ define a continuous factor map $\theta \colon X(\sigma) \to X(\tau)$, where as in the previous section $\tau$ is the Thue-Morse substitution. If $\rho \colon X(\tau) \to \Z_{(2)}$ is the continuous factor map then $\pi = \rho \circ \theta \colon X(\sigma) \to \Z_{(2)}$  is a continuous factor map.

Take the cilinder set $V=\{ x' \in X(\sigma) \mid x'(0) = 2\}$ and consider the points  $x=(\sigma^{2})^{\infty}(2 \bullet 0)$ and $y=(\sigma^{2})^{\infty}(0 \bullet 2)$. Notice that $\theta(x) = (\tau^{2})^{\infty}(0 \bullet 1)$ and $\theta(y) = (\tau^{2})^{\infty}(1\bullet 0)$, and then $\pi(x) = \rho \circ \theta (x) = \rho \circ \theta (y)  = \pi(y)$. Thus, $(x,y) \in \RP^{[1]}(X(\sigma))$.
    
    Notice that if $s(n) = (S^nx)_0$ for $n \in \N_0$ then 
\begin{equation*}
s(n) =
    \begin{cases}
        \begin{array}{ccc}
             0 & \text{ if } & t(n) =0   \\
            1 & \text{ if } & t(n) = 1 \text{ and the first digit }1 \text{ in the binary base is an odd index}\\
            2 & \text{ if } & t(n) =1 \text{ and the first digit }1 \text{ in the binary base is an even index}
        \end{array}
    \end{cases}
\end{equation*}

Finally, $V$ is a neighborhood of $y$ and $N(x,V) = \{ n \in \N \mid s(n) =2\}$. Using \cite[Lemma 3.3]{bucci_Hindman_Puzynina_Zamboni2013additive_thue_morse}, we get that there is no infinite $B \subset \N$ such that $B\oplus B \subset N(x,V)$. 
\end{proof}

\section{Combinatorial applications} \label{sec combinatorial app}

\subsection{Examples of $U^k(\Phi)$-uniform sets} \label{sec examples U^k uniform sets} 

In this section we give examples of $U^k(\Phi)$-uniform sets and therefore sets where it is possible to apply Theorem \ref{theorem uniformity BB} and the result from \Cref{sec uniform sets}. We start with two general observations.

\begin{itemize}
    \item If $A$ is $U^k(\Phi)$-uniform and  $0<\diff_\Phi(A) <1$, then its compliment $\N \setminus A$ is also $U^k(\Phi)$-uniform.
    \item If $A$ is $U^k(\Phi)$-uniform, then  the translate $A+t$ remains $U^k(\Phi)$-uniform, for any $t \in \Z$.
\end{itemize}

\begin{example*}
    Fix $\alpha \not \in \Q$, $U\subset \S^1$ an open set, $k \geq 2$ and $p(t) \in \Q[t]$ a polynomial of degree $k$, then the set
    \begin{equation*}
    A = \{ n \in \N \mid e(p(n) \alpha) \in U \}
\end{equation*}
is $U^{k}(\Phi)$-uniform, for any F\o lner $\Phi$. Moreover, if $m_{\S^1}(\S^1 \backslash U)>0$, then $A$ is not $U^{k+1}(\Phi)$-uniform. The proof is similar to Hardy field sequence case that we present in what follows.  
\end{example*}

First we introduce a lemma that was originally used in \cite{frantzikinakis_host2017arithmetic_sets} to prove uniformity of set coming from multiplicative function, see \cite[proof of Theorem 1.2 (ii)]{frantzikinakis_host2017arithmetic_sets}. The original statement uses Gowers norm instead of local uniformity seminorms, but the proof uses only algebraic manipulations of seminorms that still work in our context, so we do not include the proof.

\begin{lemma} \label{lemma Frantzikinakis and Host}
    Let $k \geq 2$ and $(a(n))_{n \in \N} \subset \S^1$ be a sequence such that $\norm{a^j}_{U^k(\Phi)} =0$ for all $j \in \Z \backslash \{0\}$ and the same F\o lner sequence $\Phi$. If $F \colon \S^1 \to \R$ is a  Riemman integrable bounded function with integral zero, then $\norm{F \circ  a}_{U^k(\Phi)} =0$. In particular, if $U \subset \S^1$ is a positive measure Borel set, then 
    \begin{equation*}
         A = \{ n \in \N \mid a(n) \in U \}
    \end{equation*}
    is a $U^{k}(\Phi)$-uniform set.
\end{lemma}

\begin{example*}

Consider $B$ the collection of equivalence classes of real valued functions $\psi \colon [c, \infty) \to \R$ for some $c>0$, where we identify two functions if they agree for all large $x > c$.
A Hardy field is a subfield of the ring $(B, +, \cdot)$ that is closed under differentiation. We denote by $\cH$ the union of all Hardy fields. An example of Hardy field functions are the logarithmic-exponential functions, that is functions defined by a finite combination of $+, -, \times , \div, \log, \exp$ with real constants, originally studied by Hardy in \cite{hardy1912properties}.

    For two functions $a,b \colon [c, \infty) \to \R$ we denote $a(t) \prec b(t)$ if $\lim_{t \to \infty} a(t)/b(t) =0$. If $\psi \in \cH$, we say that $(\psi(n))_{n \in \N}$ is a Hardy field sequence. Hardy field sequences have been extensively studied in ergodic theory, see for instance \cite{bergelson_moreira_richter2020multiplehardy,Boshernitzan94hardy,donoso_koutsogiannis_Kuca_sun_tsinas2024seminormhardy,donoso_koutsogiannis_Kuca_sun_tsinas2025resolvinghardy,donoso_koutsogiannis_sun2025jointhardy,Frantzikinakis10hardy,Frantzikinakis15hardy,Frantzikinakis_Wierdl09hardy,tsinas2023jointhardy}.

\begin{proposition}
    Let $(\psi(n))_{n \in \N}$ be a Hardy field sequence with $ t^{k} \log t \prec \psi(t) \prec t^{k+1}$ for some $k \in \N$, then for any Borel set $U\subset \S^1$, the set
    \begin{equation*}
    A = \{ n \in \N \mid e(\psi(n)) \in U \}
\end{equation*}
is $U^k (\Phi)$-uniform for $\Phi_N = \{1, \ldots, N\}$. 
\end{proposition}

\begin{proof}
    For $h_1, \ldots, h_k \in \N$ and $n \in \N$, define $\partial_{h_1, \ldots, h_k} \psi = \partial_{h_1} \cdots \partial_{h_{k}} \psi $ where $(\partial_h \psi)(n) = \psi(n+h)-\psi(n)$ for all $n \in \N$. Fix $h_1, \ldots, h_k \in \N$, by \cite[Lemma 3.1]{Frantzikinakis15hardy}, $\log t \prec (\partial_{h_1, \ldots, h_k} \psi)(t) \prec t$. By \cite[Theorem 1.3]{Boshernitzan94hardy}, $( e( \partial_{h_1, \ldots, h_k} \psi(n)))_{n \in \N}$ equidistributes in $\S^1$ and hence
    \begin{align*}
        \norm{e(\psi(n))}_{U^k(\Phi)}^{2^k} = \lim_{H \to \infty} \lim_{N \to \infty} \frac{1}{H^k} \sum_{h_1, \ldots, h_k \leq H} \frac{1}{N} \sum_{n \leq N} \Delta_{h_k} \cdots \Delta_{h_1} e(\psi(n)) \\
        = \lim_{H \to \infty} \lim_{N \to \infty} \frac{1}{H^k} \sum_{h_1, \ldots, h_k \leq H} \frac{1}{N} \sum_{n \leq N} e( \partial_{h_1, \ldots, h_k}\psi(n)) =0. 
    \end{align*}
    Notice that for all $j \neq 0$, $t^{k} \log t \prec j \cdot \psi(t) \prec t^{k+1}$ so we can repeat the same proof and conclude that $\norm{e(j \cdot \psi(n))}_{U^k(\Phi)} = \norm{e(\psi(n))^j}_{U^k(\Phi)} =0$. Thus we conclude using \cref{lemma Frantzikinakis and Host} with $a(n) = e(\psi(n))$. 
\end{proof}
    
\end{example*}

\subsection{Sumset patterns in multiple uniform sets}  \label{sec mult sets}

 In \cref{theorem uniformity BB} one can take different $U^k(\Phi)$-uniform sets $A_1, \ldots, A_k$ with $B^{\oplus i} \subset A_i$ for some infinite set $B$ and all $i =1, \ldots,k$. One would like to know if something similar happens without the $U^k(\Phi)$-uniformity assumption.
 
 From \cite[Theorem 2.6]{kousek_radic2025unrestrictedBB}, one can deduce that if for two sets $A_1, A_2 \subset \N$ there is a F\o lner sequence $\Phi$ such that $\diff_{\Phi}(A_1) + \diff_{\Phi}(A_2) > 1 $ then there exists an infinite set $B \subset \N$ and a shift $t\in \N$ such that $B \subset A_1 -t$ and $B \oplus B \subset A_2 -t$. This bound is sharp, indeed following the same steps of \cite[Proposition 4.1]{kousek_radic2025unrestrictedBB} proof one can show that taking
\begin{align*}
    A_2  &=  \bigcup_{N \in \N} [4^N, (2-1/N) \cdot 4^N ) \quad \text{ and } \\ 
    A_1 &= \{ n \in \N \mid 2n \in A_2 \} = \bigcup_{N \in \N} [2 \cdot 4^{N-1}, (1-1/2N) 4^N) 
\end{align*}
if $B \subset \N$ is a set such that $B \subset A_1 - t$ and $B \oplus B \subset A_2 - t$ for some $t \in \N$, then $B$ is finite. Notice that for the F\o lner sequence $\Phi=(\{1,\ldots, 2 \cdot 4^N\})_{N \in \N}$, $\diff_{\Phi}(A_1) +  \diff_{\Phi}(A_2) = \frac{2}{3} + \frac{1}{3} = 1$ achieving the mentioned bound. In particular, this example shows that, in general, it is not enough to have $\diff_{\Phi}(A_1) >0 $ and $ \diff_{\Phi}(A_2) >0$ in order to find an infinite $B \subset \N$ and $t \in \N$ for which $B \subset A_1 - t$ and $B \oplus B \subset A_2 - t$.

\section{Open questions} \label{sec open q}

In this section we gather questions mentioned in the core of the paper including some further background. We start with a generalization of \cite[Conjecture 3.26]{kra_Moreira_Richter_Roberson2025problems}, for which \cref{theorem uniformity BB} is a special case. 

\begin{conjecture} \label{conjecture KMRR BB}
    Let $A_1, \ldots, A_k \subset \N$ be $U^k(\Phi)$-uniform sets for the same F\o lner sequence $\Phi$. Then for every $ 1 \leq \ell_1 < \ell_2 < \cdots < \ell_k \leq \ell$, there exists an infinite set $B \subset \N$ such that for all $i=1, \ldots, k$, 
         \begin{equation*} 
             B^{\oplus \ell_i} \subset A_i. 
         \end{equation*}
\end{conjecture}

The conjecture in \cite{kra_Moreira_Richter_Roberson2025problems} is the case $A = A_1 = \dots = A_k$ of a single $U^k(\Phi)$-uniform set $A$. More generally, we could ask if the same result holds when one intersects the sets $A_1, \ldots, A_k$ with a $\nilbohrO{}$ set $D \subset \N$ as in \cref{theorem uniformity BB with nilbohr}.

One way to address this conjecture would be to upgrade the consequence of \cref{theorem mu almost every BB} from an \emph{almost surely} result to a consequence for \emph{every} points. The problem with the approach used here is that the equality \eqref{eq equality that happends mu almost surely} is only true for a set of full measure. In \cite[Appendix A]{hernandez_kousek_radic2025density_Hindman} we give an example where the equality is not achieved for a given point. 
Developing that example further, let $(\T^2, m_{\T^2}, T)$ be the ergodic affine system given by $T(x,y) = (x+\alpha, y+ 2x + \alpha)$ for all $(x,y) \in \T^2$. We have that $(\T^2, m_{\T^2}, T)$ is a $2$-step nilsystem and $Z_1 = \T$ where the factor map $\pi_1 \colon X \to Z_1 $ is given by the projection to the first coordinate. 
Therefore, the set 
$$E= \T \times (1/9^3,1/9^2) \quad \text{ is } \ U^2(\T^2, m_{\T^2}, T)\text{-uniform.}$$ 
Thus take $a=(0,0)$ and consider $\sigma_a \in \cM(X^4)$ the measure defined as in \cref{definition of sigma}. Then $\sigma_a$ is the Haar measure of
\begin{equation*}
    Y = \{ (0,0,t,s, 2t,4s,3t,9s) \in X^4 \mid t,s \in \T\} 
\end{equation*}
notice that if $s \in (1/9^3,1/9^2)$ then $9 s \in (1/9^2,1/9)$, in particular $9s \not \in (1/9^3 , 1/9^2)$. Therefore $Y \cap (X \times E \times X \times E) = \emptyset$, hence $\sigma_a(X \times E\times X \times E) =0 $. Since that measure is zero, we cannot conclude that there is an Erd\H{o}s progression in $(a,x_1,x_2,x_3)$  with $x_1 \in E$ and $x_2 \in E$, so there is no direct way to conclude that $B \cup (B \oplus B \oplus B) \subset N(a,E)$ (which should be true under \cref{conjecture KMRR BB}).

Concerning minimal systems and sumset characterizations of $\RP^{[k]}(X)$, \cref{prop counterex gen BB} shows that \cref{theorem summary} cannot be generalized to uniquely ergodic minimal systems. We wonder then 
\begin{question} \label{question RP}
    Can \cref{theorem summary} be generalized beyond pronilsystems? In particular, is it true for minimal distal systems?
\end{question}

In a different direction, we can ask about the $k$-step pronilfactor of a joining $\lambda \in J_e(\mu, \nu)$. In \cref{prop joining of X and a pronil} we prove that if $(X, \mu,T)$ is an ergodic system and $(Y, \nu,S)$ is an $s$-step pronilsystem, then $Z_k(\lambda)$ is a joining of $Z_k(\mu)$ and $Z_k(\nu)$. 

\begin{question} \label{question joinings}
    Is it true that for any ergodic systems $(X, \mu,T)$ and $(Y, \nu,S)$ and any ergodic joining $\lambda \in J_e(\mu, \nu)$, the factor $Z_k(\lambda)$ is a joining of the factors $Z_k(\mu)$ and $Z_k(\nu)$? 
\end{question}

We end this section by showing that the answer is yes when $\mu \times \nu$ ergodic.

\begin{proposition}
    Let $(X, \mu,T)$ and $(Y, \nu , S)$ be a ergodic systems such that $\mu \times \nu$ is also ergodic. Then for every $k \in \N$ $Z_k(\mu \times \nu) = Z_k (\mu ) \times Z_k(\nu)$.
\end{proposition}

\begin{proof}
    For this we prove that for every $k \in \N$ $(\mu \times \nu)^{\hkbraket{k}} = \mu^{\hkbraket{k}} \times \nu^{\hkbraket{k}}$ with the obvious identification. For the case $k =1 $ this is trivial, because by ergodicity $(\mu \times \nu)^{\hkbraket{1}} = (\mu \times \nu) \times (\mu \times \nu) = \mu^{\hkbraket{1}} \times \nu^{\hkbraket{1}}$. In particular, $Z_1(\mu \times \nu) = Z_1 (\mu) \times Z_1(\nu)$
    
    By induction, suppose that for any system whose product is ergodic we have the proposition for $k -1$. Then notice that
    \begin{align*}
        (\mu \times \nu)^{\hkbraket{k}} = \int_{Z_1(\mu \times \nu)} (\mu \times \nu)^{\hkbraket{k-1}}_s \diff m_{Z_1(\mu \times \nu)}(s)
    \end{align*}
    where $(\mu \times \nu)_s$ for $s \in m_{Z_1(\mu \times \nu)}$ is the ergodic decomposition of $(\mu \times \nu) \times (\mu \times \nu)$ over $Z_1(\mu \times \nu)$ (see \cite[Theorem 18, Chapter 4]{Host_Kra_nilpotent_structures_ergodic_theory:2018}). Now, since $Z_1(\mu \times \nu) = Z_1(\mu) \times Z_1(\nu)$ we get that $(\mu \times \nu)_s = (\mu_{s_1} \times \nu_{s_2})$ for $m_{Z_1(\mu \times \nu)}$-almost every $s = (s_1, s_2)$. Likewise, for $m_{Z_1(\mu \times \nu)}$-almost every $s = (s_1, s_2)$,  $(\mu_{s_1} \times \nu_{s_2})$ is ergodic, so by inductive assumption  
    \begin{equation*}
        (\mu \times \nu)^{\hkbraket{k-1}}_s = (\mu_{s_1} \times \nu_{s_2})^{\hkbraket{k-1}} = \mu_{s_1}^{\hkbraket{k-1}} \times \nu_{s_2}^{\hkbraket{k-1}}
    \end{equation*}
    and therefore,
    \begin{align*}
        (\mu \times &\nu)^{\hkbraket{k}} = \int_{Z_1(\mu)} \bigg( \int_{Z_1(\nu)}  \mu_{s_1}^{\hkbraket{k-1}} \times \nu_{s_2}^{\hkbraket{k-1}} \diff m_{Z_1(\nu)}(s_2) \bigg) \diff m_{Z_1(\mu)}(s_1) \\
        &= \left( \int_{Z_1(\mu)}   \mu_{s_1}^{\hkbraket{k-1}} \diff m_{Z_1(\mu)}(s_1)\right)  \times \left(\int_{Z_1(\nu)} \nu_{s_2}^{\hkbraket{k-1}} \diff m_{Z_1(\nu)}(s_2) \right) = \mu^{\hkbraket{k}} \times \nu^{\hkbraket{k}}.
    \end{align*}
    To conclude the proof, we need to check that $\norm{f \otimes g}_{U^k(X \times Y, \mu \times \nu , T \times S)} =0$ if and only if $\norm{f }_{U^k(X , \mu  , T )} =0$ or $\norm{g}_{U^k(Y, \nu ,  S)} =0$. This is direct by noticing that
    \begin{align*}
        &\norm{f \otimes g}_{U^k(X \times Y, \mu \times \nu , T \times S)}^{2^k} = \int_{(X \times Y)^{\hkbraket{k}}} \bigotimes_{\epsilon \in \hkbraket{k}} C^{|\epsilon|} (f \otimes g) \diff (\mu \times \nu)^{\hkbraket{k}} = \\
        &\left( \int_{X^{\hkbraket{k}}} \bigotimes_{\epsilon \in \hkbraket{k}} C^{|\epsilon|} f  \diff \mu^{\hkbraket{k}} \right)  \left( \int_{Y^{\hkbraket{k}}} \bigotimes_{\epsilon \in \hkbraket{k}} C^{|\epsilon|} g  \diff \nu^{\hkbraket{k}} \right)  = \norm{f }_{U^k(X , \mu , T )}^{2^k} \cdot \norm{g }_{U^k(Y , \nu , S )}^{2^k}. \qedhere
    \end{align*}
    \end{proof}

\small{
\bibliographystyle{abbrv}
\bibliography{refs}

}

\end{document}